\documentclass[11pt,reqno]{amsart}

\usepackage[T1]{fontenc}
\usepackage[utf8]{inputenc}
\usepackage[letterpaper,margin=1in]{geometry}
\usepackage{amsmath}
\usepackage{amssymb}
\usepackage{amsthm}
\usepackage[colorlinks=true,linkcolor=blue,citecolor=red,urlcolor=blue]{hyperref}

\allowdisplaybreaks
\providecommand{\newblock}{\hskip .11em plus .33em minus .07em}

\theoremstyle{plain}
\newtheorem{theorem}{Theorem}[section]
\newtheorem{corollary}[theorem]{Corollary}

\newtheorem{lemma}[theorem]{Lemma}

\theoremstyle{definition}
\newtheorem{definition}[theorem]{Definition}
\newtheorem{remark}[theorem]{Remark}

\newtheorem{example}{Example}

\title[Optimal control of nonlinear heat models]
      {Optimal control of a class of nonlinear heat conduction models}

\author[D. Menezes]{Denilson Menezes}
\address[D. Menezes]{Departamento de Ci\^encia e Tecnologia, Universidade
  Federal Rural do Semi-\'Arido, Cara\'ubas, RN, Brazil
  \newline\indent
  Chair for Dynamics, Control, Machine Learning and Numerics, Department of
  Mathematics, Friedrich--Alexander--Universit\"at Erlangen--N\"urnberg,
  Cauerstr.\ 11, 91058 Erlangen, Germany}
\email{denilson.menezes@ufersa.edu.br}

\author[J. L\'imaco]{Juan L\'imaco}
\address[J. L\'imaco]{Instituto de Matem\'atica e Estat\'istica, Universidade
  Federal Fluminense, Niter\'oi, RJ, Brazil}
\email{jlimaco@id.uff.br}



\date{\today}

\thanks{This is a preprint of an article accepted for publication in
  \emph{Evolution Equations and Control Theory}. The final published version
  will be available at the publisher's site.}

\begin{document}

\begin{abstract}
We study an optimal control problem for a quasilinear parabolic equation modeling nonlinear heat conduction during heat treatment of metals. Under general assumptions, we prove existence of an optimal control and characterize the associated optimal control-state pair in a practically relevant framework. The main difficulty is the quadratic gradient term, which we handle through energy estimates and differentiability properties of the control-to-state map.
\end{abstract}

\maketitle


\section{Introduction}


Let \( T > 0 \) be a fixed time horizon and let \( I \subset \mathbb{R} \) be a bounded open interval of length \( |I| \). Define the space-time domain \( Q := (0, T) \times I \) and its lateral boundary \( \Sigma := (0, T) \times \partial I \). In this paper, we investigate the optimal controllability of the following nonlinear parabolic partial differential equation (PDE):
\begin{equation} \label{MainModel}
    \begin{cases}
        u_t - \left(a(u) u_x \right)_x + b(u) u_x^2 = f, & \text{in } Q, \\
        u = 0, & \text{on } \Sigma, \\
        u(0,x) = u_0(x), & \text{for } x \in I,
    \end{cases}
\end{equation}
where \( f \) acts as a distributed control function. The functions \( a(\cdot) \) and \( b(\cdot) \) are state-dependent coefficients satisfying specific structural assumptions detailed below.

The primary motivation for studying this model comes from nonlinear heat conduction theory. In particular, it describes temperature evolution in certain metallic materials undergoing complex heat treatment processes \cite{rincon2006nonlinear,teixeira2009numerical}. Within specific temperature ranges, the physical properties of these materials (such as thermal conductivity and heat capacity) exhibit significant dependence on temperature.

To derive \eqref{MainModel} from physical principles, let \( \theta(t,x) \) denote the temperature field. The conservation of energy reads:
\[
\rho(\theta) \dot{e} + q_x = \widetilde{f},
\]
where \( \rho \) is density, \( e \) is internal energy density, \( q = -\kappa(\theta)\theta_x \) is the heat flux (Fourier's law), and \( \widetilde{f} \) is a heat source. Assuming \( e = e(\theta) \) with \( c(\theta) = \partial e/\partial\theta > 0 \), and introducing \( u = e(\theta) \), \( \alpha(u) = \kappa(\theta)/c(\theta) \), and a source term of the form \( \widetilde{f}(t,x,u) = \rho(u)f(t,x) \), one obtains \eqref{MainModel} with
\[
a(u) = \frac{\alpha(u)}{\rho(u)}, \qquad b(u) = -\frac{\alpha(u)}{\rho^2(u)}\rho'(u).
\]
The homogeneous Dirichlet boundary condition is consistent with the references \cite{rincon2006nonlinear,teixeira2009numerical}. In typical heat-treatment scenarios for metals, one expects $p$ and $q$ in the growth conditions below to be moderate (often $p = q = 2$).

We note that the quasilinear equation in non-divergence form
\[
u_t - a(u) u_{xx} = f
\]
can be transformed into \eqref{MainModel} via the identity
\[
a(u) u_{xx} = (a(u) u_x)_x - a'(u) u_x^2,
\]
which aligns with our framework under appropriate assumptions. While non-divergence forms are sometimes preferred for parameter calibration \cite{cannon1973determining,borukhov2005functional,ramalingam2006accurate,vala2015computational}, the divergence form \eqref{MainModel} is more natural for variational analysis and control theory.


Optimal control of quasilinear parabolic equations has been studied extensively; see, e.g., \cite{casas1995optimal, abdulla2019optimal, bonifacius2018second, casas2018analysis}. The work \cite{casas2018analysis} treats a broad class of nonlinearities using truncation methods. However, the presence of the quadratic gradient term \( b(u) u_x^2 \) in \eqref{MainModel} introduces significant analytical challenges. Unlike lower-order or state-dependent nonlinearities, the quadratic gradient term couples the state and its spatial derivative nonlinearly, which substantially affects the variational structure and regularity requirements of the problem. In particular, truncation techniques are unable to guarantee coercivity or monotonicity of the associated operator when applied to such terms, thereby preventing the direct application of the methods and existence theory developed in \cite{casas2018analysis}. To the best of our knowledge, no existing result covers this specific structure, and alternative approaches are required to ensure well-posedness and optimal control in models exhibiting this nonlinear gradient coupling.

Related studies in fluid mechanics have established fundamental results on the controllability of quasilinear models governing non-Newtonian fluids \cite{arada2012optimal, arada2013optimal}, and substantial progress has been made on exact and null controllability for parabolic systems \cite{fernandez2000null, doubova2002controllability, fernandez2017theoretical, de2023local, fernandez2023local}; for an overview of the broader landscape we refer to the survey \cite{zuazua2007controllability}. Particularly close in spirit to our setting is \cite{doubova2002controllability}, which addresses controllability for parabolic systems with nonlinearities involving both the state and the gradient---a structural feature shared with~\eqref{MainModel}---although the goal there is to drive the state to a target rather than to minimize a quadratic cost. Recent advances in the optimal control theory of quasilinear parabolic equations incorporating gradient nonlinearities include \cite{bonifacius2025optimal}, which tackles higher-dimensional configurations $(d = 2, 3)$ through pointwise constraints on the state gradient to regulate the growth of quadratic gradient terms and prevent finite-time blow-up. While this framework imposes explicit gradient-state constraints on the admissible set, our work complements these developments by adopting a distinct perspective: we analyze the one-dimensional case under data smallness conditions and provide a complete characterization of optimality through a coupled forward-backward system, thereby avoiding the necessity of incorporating gradient constraints into the control admissibility conditions. Our principal contribution lies in establishing a rigorous theory for \emph{optimal distributed control} of the model \eqref{MainModel} with quadratic gradient nonlinearities---a structure that captures essential phenomena in nonlinear thermal conduction during materials heat treatment processes.

We consider the following optimal control problem:
\begin{equation} \label{OptProblem}
    \min_{f \in \mathcal{U}_{ad}} J(f), \quad \text{where} \quad J(f) = \frac{1}{2} \| Su(f) - \Xi_d \|_V^2 + \frac{1}{2} (N f, f)_{L^2(Q)}.
\end{equation}
Here, \( \mathcal{U}_{ad} \subset H^1(0,T; L^2(I)) \) is a closed convex set of admissible controls, \( u(f) \) is the state corresponding to control \( f \) (solution of \eqref{MainModel}), \( S: W \to V \) is a continuous linear observation operator mapping the state space \( W \) (defined in Section~\ref{sec:WPofModel}) to a Hilbert space \( V \), \( \Xi_d \in V \) represents desired target data, and \( N: L^2(Q) \to L^2(Q) \) is a positive definite operator modeling the control cost. This formulation encompasses both trajectory tracking and source identification problems \cite{lions1971optimal}.

\begin{remark}[On the choice of the control space] \label{rem:choice_control_space}
The natural framework for distributed control of \emph{linear} parabolic equations is $f \in L^2(Q)$ with pointwise constraints; see, e.g., \cite{lions1971optimal,casas1995optimal}. The quadratic gradient nonlinearity $b(u)u_x^2$ in \eqref{MainModel}, however, demands additional time-regularity of the source: the well-posedness theory of Section~\ref{sec:WPofModel} relies on bounding $u_t$ via the time-differentiated equation, in which $f_t$ appears as a source. Concretely, our energy estimates for $u_t$ in $L^2(0,T;H^1_0(I))$ require $f_t \in L^2(Q)$, and Estimate IV of Theorem~\ref{thm:gateaux_differentiability} requires the analogous bound on the perturbation $h_t = (g-f)_t$. Working with $\mathcal{U}_{ad} \subset H^1(0,T;L^2(I))$ is therefore not a matter of mathematical convenience but a structural requirement of the nonlinearity; we comment further on this in Remark~\ref{rem:strong_solution_choice} below.
\end{remark}

Our analysis relies on the following technical assumptions.

\begin{enumerate}
    \item[(\textbf{A1})] \label{A1} \( a \in C^2(\mathbb{R}) \), and there exist constants \( a_0 > 0 \), \( M > 0 \), and \( p \geq 2 \) such that
    \[
    a(s) \geq a_0, \quad a(s) \leq M(1 + |s|^p), \quad |a'(s)| \leq M(1 + |s|^{p-1}), \quad |a''(s)| \leq M(1 + |s|^{p-2}).
    \]
    \item[(\textbf{A2})] \label{A2} \( b \in C^2(\mathbb{R}) \), and there exist \( M > 0 \) and \( q \geq 1 \) such that
    \[
    \begin{gathered}
    b(s)s \geq 0, \qquad |b(s)| \leq M(1 + |s|^q), \\
    |b'(s)| \leq M(1 + |s|^{(q-1)_+}), \qquad |b''(s)| \leq M(1 + |s|^{(q-2)_+}),
    \end{gathered}
    \]
    where $r_+ := \max(r,0)$.
    \item[(\textbf{A3})] \label{A3} The initial data and control satisfy
    \[
    u_0 \in H^2(I) \cap H^1_0(I), \quad f, f_t \in L^2(Q),
    \]
    and there exists \( \delta > 0 \) such that
    \[
    \|u_0\|_{H^2(I)}^2 + \|f\|_{L^2(Q)}^2 + \|f_t\|_{L^2(Q)}^2 \leq \delta^2.
    \]
\end{enumerate}

\begin{example}\label{ex:nonlinearities}
The functions \( a(s) = s^p + 1 \) (with \( p \) even, $p \geq 2$) and \( b(s) = s^q \) (with \( q \) odd, $q \geq 1$) satisfy (\textbf{A1}) and (\textbf{A2}). Indeed, for $a(s) = s^p + 1$ with $p$ even one has $a(s) \geq 1 > 0$, $|a(s)| \leq 1 + |s|^p \leq 2(1 + |s|^p)$, and the bounds on $a'$, $a''$ follow directly. For $b(s) = s^q$ with $q$ odd, $b(s)s = s^{q+1} \geq 0$ and the growth bounds are immediate; the simplest case $q = 1$ corresponds to $b(s) = s$, for which $|b'(s)| = 1$ and $|b''(s)| = 0$ trivially fit the (\textbf{A2}) bounds with $(q-1)_+ = (q-2)_+ = 0$. A second physically motivated example is $a(s) = a_0 + a_1 s^2$ and $b(s) = b_0 s$, which appears in models of thermal conductivity weakly dependent on temperature.
\end{example}

The sign condition \( b(s)s \geq 0 \) in (\textbf{A2}) is physically motivated and ensures the energy-type estimates in Appendix~\ref{A}. The polynomial growth conditions are technical but permit a wide class of nonlinearities.


The analysis of the quadratic gradient nonlinearity $b(u)u_{x}^{2}$ poses significant analytical challenges, particularly in obtaining energy estimates and proving the differentiability of the control-to-state mapping. To address these challenges in a setting where the optimality system can be characterized completely, we restrict our analysis to the one-dimensional spatial case ($d=1$). This restriction is technically crucial, as it allows us to utilize strong Sobolev embedding results (such as $H^{1}(I)\hookrightarrow L^{\infty}(I)$) which are fundamental for controlling the nonlinear terms within the energy estimates and for establishing the regularity required by the differentiability analysis. The extension of this framework to $d=2$ and $d=3$ is a natural and important open problem, which we discuss further in Section~\ref{sec:conclusions}.

The main contributions of this work are: a local well-posedness result for the state equation \eqref{MainModel} under assumptions (\textbf{A1})--(\textbf{A3}), extending the analysis of \cite{fernandez2023local} to the controlled case; a detailed study of the control-to-state mapping \( f \mapsto u(f) \), including differentiability in appropriate function spaces; an existence theorem for the optimal control problem \eqref{OptProblem}; and a characterization of the optimal control via a first-order necessary condition and a coupled forward-backward PDE system.

The paper is organized as follows: Section~\ref{sec:WPofModel} defines the solution concept and establishes well-posedness of the state equation; Section~\ref{sec:ContToStateMapping} analyzes the control-to-state map and its differentiability; Section~\ref{sec:optimal_control_existence} formulates the optimal control problem and proves existence of an optimal control; Section~\ref{sec:characterization} provides a characterization of the optimal control under additional assumptions; and Section~\ref{sec:conclusions} contains concluding remarks and discusses future directions.


Throughout, \( \mathcal{U}_{ad} \) denotes a closed convex subset of \( H^1(0,T; L^2(I)) \). The \( L^2(I) \) inner product and norm are written as \( (\cdot, \cdot) \) and \( \|\cdot\| \), respectively. We denote by $\partial_p Q := \Sigma \cup (\{0\}\times I)$ the parabolic boundary of $Q$. The symbol \( C \) denotes a generic positive constant that may change from line to line but depends only on the problem data (\( a, b, p, q, T, I, \text{etc.} \)); when an estimate's constant depends on additional quantities, this dependence will be made explicit.

\section{Well-posedness of the model} \label{sec:WPofModel}


We begin by introducing the functional framework for our analysis.

\begin{definition}[Solution space] \label{def:solution_space}
The space $W$ is defined as
\begin{equation} \label{defW}
    W := \left\{ u \in L^2\left(0,T; H^2(I) \cap H^1_0(I)\right) : u_t \in L^2\left(0,T; H^1_0(I)\right) \right\},
\end{equation}
equipped with the norm
\begin{equation} \label{normW}
    \|u\|_W := \left( \|u\|_{L^2(0,T;H^2(I))}^2 + \|u_t\|_{L^2(0,T;H^1_0(I))}^2 \right)^{1/2}.
\end{equation}
\end{definition}

\begin{remark} \label{rem:banach_space}
The space $(W, \|\cdot\|_W)$ is a Hilbert space when equipped with the inner product induced by \eqref{normW}. Moreover, the embedding $W \hookrightarrow C([0,T]; H^1_0(I))$ is continuous; in particular, the trace $u(T,\cdot) \in H^1_0(I)$ is well-defined for every $u \in W$.
\end{remark}

We will also need the following companion space, to which we refer for the weak formulation of the linearized systems:
\begin{equation} \label{defH}
    H := L^2\left(0,T;H^1_0(I)\right) \cap H^1\left(0,T;H^{-1}(I)\right),
\end{equation}
endowed with its natural Hilbertian norm. By the Aubin--Lions--Simon Theorem \cite{lions1969quelques,aubin1963analyse,simon1986compact}, the embeddings
\[
W \hookrightarrow H, \qquad H \hookrightarrow C\left([0,T]; L^2(I)\right)
\]
hold; the second embedding is compact.

We now define what we mean by a strong solution to problem \eqref{MainModel}.

\begin{definition}[Strong solution] \label{def:strong_solution}
Given $f \in H^1(0,T; L^2(I))$ and $u_0 \in H^2(I) \cap H^1_0(I)$, a function $u: Q \to \mathbb{R}$ is called a \emph{strong solution} of \eqref{MainModel} if:
\begin{enumerate}
    \item[(a)] $u \in W$;
    \item[(b)] The equation
    \begin{equation} \label{eq_strong_form}
        u_t - (a(u) u_x)_x + b(u) u_x^2 = f
    \end{equation}
    holds in the $L^2(Q)$ sense;
    \item[(c)] The initial condition $u(0,x) = u_0(x)$ is satisfied for almost every $x \in I$.
\end{enumerate}
\end{definition}

\begin{remark}[On the strong-solution space $W$] \label{rem:strong_solution_choice}
The choice $u_t \in L^2(0,T;H^1_0(I))$ in \eqref{defW}---rather than the more familiar $u_t \in L^2(Q)$---is dictated by the quadratic gradient nonlinearity $b(u)u_x^2$. Indeed, the differentiability analysis of Section~\ref{sec:ContToStateMapping} relies on testing the linearized equation with $y_{\epsilon t}$ and integrating by parts in space, which produces a coercive term in $\|y_{\epsilon xt}\|$; closing this estimate requires precisely $u_t, u_{\epsilon t} \in L^2(0,T;H^1_0(I))$, since $\|u_t(t,\cdot)\|_{L^\infty(I)}$ then becomes locally integrable in $t$ by the one-dimensional embedding $H^1_0(I) \hookrightarrow L^\infty(I)$. Working in the weaker space $L^2(Q)$ for $u_t$ would leave $\|u_t\|_{L^\infty(I)}$ undefined a priori and prevent the chain of energy estimates from closing. The same regularity is what propagates through the linearized adjoint analysis of Section~\ref{sec:characterization}; this is also why the control space carries one time-derivative (Remark~\ref{rem:choice_control_space}).
\end{remark}


The main result of this section establishes the well-posedness under small data of problem \eqref{MainModel}.

\begin{theorem}[Well-posedness under small data] \label{thm:well_posedness}
Assume that hypotheses (\textbf{A1})--(\textbf{A3}) are satisfied. There exists $\delta_0 > 0$, depending only on $a_0, M, T, |I|, p, q$, such that, if $\delta < \delta_0$ in (\textbf{A3}), then the system \eqref{MainModel} admits a unique strong solution $u \in W$, and this solution satisfies the estimate
\begin{equation} \label{eq:energy_estimate}
    \|u\|_W \leq C\left( \|u_0\|_{H^2(I)} + \|f\|_{L^2(Q)} + \|f_t\|_{L^2(Q)} \right),
\end{equation}
with a constant $C > 0$ depending only on $a_0$, $M$, $T$, $|I|$, $p$, $q$. An explicit (although non-optimal) admissible threshold $\delta_0$ can be read off the continuation argument in Appendix~\ref{A}; see in particular the smallness condition imposed on $\|f\|^2_{L^2(Q)} + \|f_t\|^2_{L^2(Q)}$ there.
\end{theorem}

\begin{proof}
The proof employs the Galerkin method combined with energy estimates. We outline the main strategy below, with complete details provided in Appendix~\ref{A}.

We construct finite-dimensional approximations using the eigenfunctions $\{\psi_j\}_{j=1}^\infty$ of the Dirichlet Laplacian on $I$. For each $n \in \mathbb{N}$, we seek approximate solutions of the form
\[
u_n(t,x) = \sum_{j=1}^n g_{j,n}(t) \psi_j(x)
\]
that satisfy a projected version of the original equation on the finite-dimensional subspace $V_n = \mathrm{span}\{\psi_1, \dots, \psi_n\}$.

The core of the proof lies in deriving uniform estimates independent of the approximation parameter $n$. We establish these through a sequence of energy estimates: first by testing with $u_n$ itself to obtain basic $L^2(I)$ bounds; then with $u_{n,t}$ (combined with a time-differentiated version of the equation) to control time derivatives; and finally with $-u_{n,xx}$ to gain spatial regularity. Crucially, the smallness of the data prevents finite-time loss of regularity, via a continuation argument that traps $\|u_{n,x}(t,\cdot)\|$ inside a small ball for all $t \in [0,T]$. These estimates yield the uniform bounds
\begin{align*}
    \|u_n\|_{L^\infty(0,T;H^1_0(I))} &\leq C, \\
    \|u_{n,t}\|_{L^2(0,T;H^1_0(I))} + \|u_n\|_{L^2(0,T;H^2(I))} &\leq C.
\end{align*}

With these uniform bounds in hand, the Aubin--Lions compactness lemma yields a subsequence $\{u_{n_k}\}$ that converges weakly in $W$ and strongly in $L^2(0,T;H^1(I))$ to a limit function $u$. The strong convergence in $L^2(0,T;H^1)$, combined with the $L^\infty(Q)$ bound on $u_n$ (via the embedding $W \hookrightarrow L^\infty(Q)$ in 1D) and the continuity of $a$, $b$, suffices to pass to the limit in the nonlinear terms $a(u_n)u_{n,x}$ and $b(u_n)u_{n,x}^2$. The regularity $u \in W$ ensures that $u$ is a strong solution.

Uniqueness is proved by considering the difference $w = u - v$ of two potential solutions. A careful energy estimate for $w$, combined with Gronwall's inequality, shows that $w$ vanishes identically.
\end{proof}


For the optimal control analysis in subsequent sections, we require additional smallness conditions on the admissible controls and initial data.

\begin{enumerate}
\item[(\textbf{A4})] [Control constraints]
The set of admissible controls $\mathcal{U}_{ad}$ satisfies, for some $\alpha \in (0,1)$,
\[
\mathcal{U}_{ad} \subseteq
\left\{\, f \in H^1\!\left(0,T;L^2(I)\right) \ \middle|\ 
\begin{aligned}
\|f(t,\cdot)\|_{L^\infty(I)} &\leq \frac{\alpha\,\delta}{2\sqrt{T}},\\[6pt]
\|f_t(t,\cdot)\|_{L^\infty(I)} &\leq \frac{(1-\alpha)\,\delta}{2\sqrt{T}}
\end{aligned}
\,\right\},
\]

where $\delta > 0$ is the constant from assumption (\textbf{A3}).

\item[(\textbf{A5})] [Initial data smallness]
The initial data satisfies
\[
\|u_0\|_{H^2(I)} \leq \frac{\delta}{2}.
\]
\end{enumerate}

\begin{remark}[Consistency of the smallness conditions] \label{rem:implication_assumptions}
Assumptions (\textbf{A4}) and (\textbf{A5}) imply
\[
\|f\|_{L^2(Q)}^2 \leq T  \cdot \frac{\alpha^2 \delta^2}{4 T } = \frac{\alpha^2 \delta^2}{4},
\qquad
\|f_t\|_{L^2(Q)}^2 \leq \frac{(1-\alpha)^2 \delta^2}{4},
\]
so that, together with $\|u_0\|_{H^2(I)}^2 \leq \delta^2/4$ from (\textbf{A5}),
\begin{equation} \label{eq:smallness_consequence}
\|u_0\|_{H^2(I)}^2 + \|f\|_{L^2(Q)}^2 + \|f_t\|_{L^2(Q)}^2 \leq \frac{\delta^2}{4}\big(1 + \alpha^2 + (1-\alpha)^2\big) \leq \delta^2,
\end{equation}
which is exactly the (\textbf{A3})-bound. Provided the parameter $\delta > 0$ is chosen so that $\delta < \delta_0$, the bound \eqref{eq:smallness_consequence} guarantees the applicability of Theorem~\ref{thm:well_posedness} for every $f \in \mathcal{U}_{ad}$. In particular, the control-to-state mapping $f \mapsto u(f)$ is well-defined on $\mathcal{U}_{ad}$, and Theorem~\ref{thm:well_posedness} provides the uniform bound $\|u(f)\|_W \leq C\delta$ for all $f \in \mathcal{U}_{ad}$, where $C$ depends only on the problem data.
\end{remark}

\begin{remark}[Regularity and embeddings] \label{rem:regularity_embeddings}
For $u \in W$, the following embeddings hold: $u \in C([0,T]; H^1_0(I))$, $u \in L^\infty(Q)$, and $u_x \in L^\infty(0,T; L^2(I)) \cap L^2(0,T; H^1(I))$. In particular, by the one-dimensional Sobolev embedding $H^1(I)\hookrightarrow L^\infty(I)$, we have $u_x \in L^2(0,T;L^\infty(I))$ and $\|u_x(t,\cdot)\|_{L^\infty(I)} \leq C\|u_{xx}(t,\cdot)\|$ for a.e.\ $t \in (0,T)$. These regularity properties will be crucial for the differentiability analysis in Section~\ref{sec:ContToStateMapping}.
\end{remark}

\section{Differentiability of the control-to-state mapping} \label{sec:ContToStateMapping}


In this section, we analyze the differentiability properties of the control-to-state mapping
\[
\mathcal{F}: \mathcal{U}_{ad} \to W, \quad \mathcal{F}(f) = u(f),
\]
where $u(f)$ is the unique strong solution of \eqref{MainModel} guaranteed by Theorem~\ref{thm:well_posedness}. Establishing the G\^ateaux differentiability of $\mathcal{F}$ is essential for deriving first-order necessary optimality conditions and developing numerical optimization algorithms.


We begin by studying the linearized system associated with the state equation. The formal linearization of $u_t - (a(u)u_x)_x + b(u)u_x^2 = f$ around $u$ in the direction $y$ produces
\begin{align*}
&y_t - \big(a(u)y_x + a'(u)u_x y\big)_x + 2b(u)u_x y_x + b'(u)u_x^2 y \\
&\qquad = y_t - (a(u)y)_{xx} + 2b(u)u_x y_x + b'(u)u_x^2 y,
\end{align*}
where the second identity uses
\(
(a(u)y)_{xx} = (a'(u)u_x y + a(u)y_x)_x.
\)
We will work with the second (more compact) form.

\begin{lemma}[Well-posedness of the linearized system] \label{lem:linearized_system}
Let $u \in W$ be the solution from Theorem~\ref{thm:well_posedness} corresponding to control $f \in \mathcal{U}_{ad}$, and let $h \in H^1(0,T; L^2(I))$. Then the linearized system
\begin{equation} \label{linearized_system}
    \begin{cases}
        y_t - (a(u)y)_{xx} + 2b(u)u_x y_x + b'(u)u_x^2 y = h, & \text{in } Q, \\
        y = 0, & \text{on } \Sigma, \\
        y(0,x) = 0, & \text{for } x \in I,
    \end{cases}
\end{equation}
admits a unique strong solution $y \in W$ satisfying the estimate
\begin{equation} \label{linearized_estimate}
    \|y\|_W \leq C\|h\|_{H^1(0,T;L^2(I))},
\end{equation}
where $C > 0$ depends on $\|u\|_W$. Moreover, the following sharper estimate holds with $h \in L^2(Q)$ alone (without requiring $h_t \in L^2(Q)$):
\begin{equation} \label{linearized_estimate_L2}
    \|y\|_{L^\infty(0,T;H^1_0(I))} + \|y\|_{L^2(0,T;H^2(I))} + \|y_t\|_{L^2(Q)} \leq C\|h\|_{L^2(Q)}.
\end{equation}
In particular, by the one-dimensional embedding $L^\infty(0,T;H^1_0(I))\hookrightarrow L^\infty(Q)$,
\begin{equation} \label{linearized_estimate_Linf}
    \|y\|_{L^\infty(Q)} \leq C\|h\|_{L^2(Q)}.
\end{equation}
\end{lemma}

\begin{proof}
The proof employs energy estimates and compactness arguments. We provide the main ideas here and refer to Appendix~\ref{B} for complete details.

We work with the weak formulation in the space $H$ defined in \eqref{defH} of Section~\ref{sec:WPofModel}, namely we seek $y \in H$ with $y|_{\partial_p Q} = 0$ such that, for all $\phi \in H$,
\begin{align*}
    \int_0^T \langle y_t, \phi \rangle \, dt &+ \int_Q \left[a(u)y_x + a'(u)u_x y\right] \phi_x \, d(t,x) \\
    &+ \int_Q \left[2b(u)u_x y_x + b'(u)u_x^2 y\right] \phi \, d(t,x) = \int_Q h\phi \, d(t,x).
\end{align*}

The existence of a solution $y \in H$ follows from the Galerkin method (see Appendix~\ref{B}). The improved regularity $y \in W$ is established through a sequence of higher-order energy estimates: testing with $y$ and $-y_{xx}$ yields the bound \eqref{linearized_estimate_L2} (which only requires $h \in L^2(Q)$), while testing the time-differentiated equation with $y_t$ produces the additional bound on $y_t$ in $L^2(0,T;H^1_0(I))$ that requires $h_t \in L^2(Q)$. The coefficients $a(u)$, $b(u)$, and $u_x$ inherit sufficient regularity from $u \in W$ to close all these estimates. Uniqueness follows via Gronwall's inequality applied to the difference of two solutions.
\end{proof}

\begin{remark}
The coefficients in \eqref{linearized_system} depend on $u$ and $u_x$. Although $u_x$ is only known a priori to belong to $L^\infty(0,T;L^2(I)) \cap L^2(0,T;L^\infty(I))$, the one-dimensional structure of the problem and the regularity $u \in W$ are enough to close the energy estimates and establish well-posedness.
\end{remark}


The main result of this section establishes the G\^ateaux differentiability of the control-to-state mapping.

\begin{theorem}[G\^ateaux differentiability] \label{thm:gateaux_differentiability}
The control-to-state mapping $\mathcal{F}: \mathcal{U}_{ad} \to W$ is G\^ateaux differentiable. Specifically, for any $f, g \in \mathcal{U}_{ad}$, the directional derivative
\[
\mathcal{F}'(f)(g - f) := \lim_{\epsilon \to 0^+} \frac{u(f + \epsilon(g - f)) - u(f)}{\epsilon}
\]
exists in $W$ and coincides with the solution $y \in W$ of the linearized system \eqref{linearized_system} with $h = g - f$. Moreover,
\[
\|\mathcal{F}'(f)(g - f)\|_W \leq C\|g - f\|_{H^1(0,T;L^2(I))},
\]
and, by the sharper bound \eqref{linearized_estimate_L2}--\eqref{linearized_estimate_Linf},
\begin{equation} \label{Fprime_L2_bound}
\begin{aligned}
&\|\mathcal{F}'(f)(g-f)\|_{L^\infty(Q)} + \|\mathcal{F}'(f)(g-f)\|_{L^\infty(0,T;H^1_0(I))} \\
&\quad + \|\mathcal{F}'(f)(g-f)\|_{L^2(0,T;H^2(I))} + \|(\mathcal{F}'(f)(g-f))_t\|_{L^2(Q)} \leq C\|g-f\|_{L^2(Q)}.
\end{aligned}
\end{equation}
\end{theorem}

\begin{proof}
Let \( 0 < \epsilon \leqslant 1 \) and  \( h = g - f \). Set \( u_\epsilon := u(f + \epsilon h) \) and \( u := u(f) \). Consider the difference quotient
\[
y_\epsilon := \frac{u_\epsilon - u}{\epsilon}.
\]
From the state equation \eqref{MainModel}, the algebraic decomposition
\begin{align*}
b(u_\epsilon)u_{\epsilon x}^2 - b(u)u_x^2 &= \big(b(u_\epsilon)-b(u)\big)u_{\epsilon x}^2 + b(u)\big(u_{\epsilon x}^2 - u_x^2\big)\\
&= \big(b(u_\epsilon)-b(u)\big)u_{\epsilon x}^2 + b(u)(u_{\epsilon x}+u_x)(u_{\epsilon x}-u_x),
\end{align*}
and the analogous identity for the diffusion term, the function \( y_\epsilon \) satisfies the system
\begin{equation} \label{systemYeps}
\begin{cases}
    y_{\epsilon t} - (a(u_\epsilon) y_{\epsilon x})_x + b(u)(u_{\epsilon x} + u_x) y_{\epsilon x} \\
    \quad - \dfrac{1}{\epsilon}\big[(a(u_\epsilon)-a(u)) u_x\big]_x + \dfrac{1}{\epsilon}(b(u_\epsilon)-b(u)) u_{\epsilon x}^2 = h, & \text{in } Q, \\
    y_\epsilon = 0, & \text{on } \Sigma, \\
    y_\epsilon(0,x) = 0, & \text{for } x \in I.
\end{cases}
\end{equation}

\medskip\noindent\textbf{CLAIM 1.} There exists a constant $C>0$, depending only on the data and the model parameters, such that
\[
\|y_\epsilon\|_W^2 \leq C\,\|h\|_{H^1(0,T;L^2(I))}^2 \qquad \text{for every } \epsilon \in (0,1].
\]
Here $C$ depends on $\|f\|_{H^1(0,T;L^2(I))}$, $\|g\|_{H^1(0,T;L^2(I))}$, $\|u_0\|_{H^2(I)}$, and the model parameters.

We prove this claim via a sequence of energy estimates. Throughout, generic constants $C > 0$ may depend on the indicated quantities. We will frequently use the following bounds, which follow from the well-posedness theory and the Sobolev embedding $H^1(I) \hookrightarrow L^\infty(I)$:
\begin{align}
    \|u\|_{L^\infty(Q)} &\leqslant C \left(\|f\|_{H^1(0,T;L^2(I))} + \|u_0\|_{H^2(I)}\right), \label{appDtrivial_1} \\
    \|u_\epsilon\|_{L^\infty(Q)} &\leqslant C\left(\|f\|_{H^1(0,T;L^2(I))} + \|g\|_{H^1(0,T;L^2(I))} + \|u_0\|_{H^2(I)} \right). \label{appDtrivial_2}
\end{align}

Some immediate consequences of these bounds and (\textbf{A1})--(\textbf{A2}) are the following uniform pointwise estimates. There exists a constant $C > 0$, depending only on the data and on a uniform bound for $\|f\|_{H^1(0,T;L^2(I))}$, $\|g\|_{H^1(0,T;L^2(I))}$, $\|u_0\|_{H^2(I)}$, such that:
\begin{itemize}
    \item[(i)] \(\max\{|a(u)|,\,|a(u_\epsilon)|\} \leqslant C\), \(\max\{|b(u)|,\,|b(u_\epsilon)|\} \leqslant C\) (where the right-hand side is the new $C$, obtained from $C(1+\|u\|_{L^\infty(Q)}^p)$ and the analogous bound for $u_\epsilon$, both of which are finite by \eqref{appDtrivial_1}--\eqref{appDtrivial_2}).
    \item[(ii)] For all \( 0 \leqslant s \leqslant 1 \), defining \( u^{s,\epsilon} := s(u_\epsilon - u) + u \), one has $u^{s,\epsilon} \in L^\infty(Q)$ with $\|u^{s,\epsilon}\|_{L^\infty(Q)} \leq C$, and therefore
    \[
        \max\big\{|a'(u^{s,\epsilon})|,\, |a''(u^{s,\epsilon})|,\, |b'(u^{s,\epsilon})|,\, |b''(u^{s,\epsilon})|\big\} \leqslant C.
    \]
\end{itemize}
With a slight abuse, throughout the rest of this proof we use $C$ to denote any positive constant of this kind, possibly changing from line to line.

We will also use repeatedly the one-dimensional Sobolev embedding $H^1_0(I) \hookrightarrow L^\infty(I)$ in the form
\begin{align*}
\|v\|_{L^\infty(I)} &\leq C\|v_x\| && \forall\, v \in H^1_0(I),\\
\|v_x\|_{L^\infty(I)} &\leq C\|v_{xx}\| && \forall\, v \in H^2(I)\cap H^1_0(I).
\end{align*}
The second estimate follows from $v_x \in H^1(I)$ together with Poincaré's inequality. We will also need the following \emph{sharper} one-dimensional Gagliardo--Nirenberg interpolation, valid for $v \in H^2(I)\cap H^1_0(I)$:
\begin{equation} \label{GN_sharp}
\|v_x\|_{L^\infty(I)}^2 \;\leq\; C\,\|v_{xx}\|\,\|v_x\|,
\end{equation}
which we will combine below with the uniform-in-$t$ bound $\|u_x(t,\cdot)\| + \|u_{\epsilon,x}(t,\cdot)\| \leq C$ (provided by the embedding $W \hookrightarrow C([0,T]; H^1_0(I))$ together with the bounds \eqref{appDtrivial_1}--\eqref{appDtrivial_2}) to obtain
\begin{equation} \label{GN_sharp_uniform}
\|u_x(t,\cdot)\|^2_{L^\infty(I)} \leq C\,\|u_{xx}(t,\cdot)\|,\qquad
\|u_{\epsilon,x}(t,\cdot)\|^2_{L^\infty(I)} \leq C\,\|u_{\epsilon,xx}(t,\cdot)\|.
\end{equation}
This refined bound is essential to keep the time-coefficients in our Gronwall arguments in $L^1(0,T)$.

\begin{center}
    \textbf{ESTIMATE I}
\end{center}

Using \( y_\epsilon \) as a test function in \eqref{systemYeps}, integrating over \( I \), and performing integration by parts, we obtain
\begin{equation} \label{FirstEstimateYeps}
\begin{aligned}
\frac{1}{2} \frac{d}{dt} \|y_\epsilon\|^2 &+ a_0 \|y_{\epsilon x}\|^2 \leqslant I_1 + I_2 + I_3 + \int_I |h| |y_\epsilon| \, dx \\
&\leqslant I_1 + I_2 + I_3 + \frac{1}{2} \|y_\epsilon\|^2 + \frac{1}{2} \|h\|^2,
\end{aligned}
\end{equation}
where
\begin{align*}
I_1 &:= \int_I |b(u)| |u_{\epsilon x} + u_x| |y_{\epsilon x}| |y_\epsilon| \, dx, \\
I_2 &:= \int_I \frac{|a(u_\epsilon) - a(u)|}{\epsilon} |u_x| |y_{\epsilon x}| \, dx, \\
I_3 &:= \int_I \frac{|b(u_\epsilon) - b(u)|}{\epsilon} u_{\epsilon x}^2 |y_\epsilon| \, dx.
\end{align*}

By the Mean Value Theorem and (ii) above, we have
\begin{equation} \label{MVTbounds}
\frac{|a(u_\epsilon)-a(u)|}{\epsilon} \leq C|y_\epsilon|, \qquad \frac{|b(u_\epsilon)-b(u)|}{\epsilon} \leq C|y_\epsilon|.
\end{equation}

Using \eqref{MVTbounds} and the embeddings, we estimate each term. First,
\begin{align}
|I_1| &\leqslant C \|u_{\epsilon x} + u_x\|_{L^\infty} \|y_{\epsilon x}\| \|y_\epsilon\| \nonumber \\
      &\leqslant C \big(\|u_{\epsilon xx}\| + \|u_{xx}\|\big) \|y_{\epsilon x}\| \|y_\epsilon\| \nonumber \\
      &\leqslant \epsilon_1 \|y_{\epsilon x}\|^2 + C_{\epsilon_1} \big(\|u_{\epsilon xx}\|^2 + \|u_{xx}\|^2\big) \|y_\epsilon\|^2, \label{I_1EstYeps}
\end{align}
where $\epsilon_1 > 0$ is to be chosen. For $I_2$, using \eqref{MVTbounds} we have
\begin{align}
|I_2| &\leqslant C \int_I |y_\epsilon||u_x||y_{\epsilon x}|\, dx \leqslant C \|u_x\|_{L^\infty} \|y_{\epsilon x}\| \|y_\epsilon\| \nonumber \\
      &\leqslant C \|u_{xx}\| \|y_{\epsilon x}\| \|y_\epsilon\| \nonumber \\
      &\leqslant \epsilon_1 \|y_{\epsilon x}\|^2 + C_{\epsilon_1} \|u_{xx}\|^2 \|y_\epsilon\|^2. \label{I_2EstYeps}
\end{align}
For $I_3$, using \eqref{MVTbounds},

\begin{align}
|I_3| &\leqslant C \int_I u_{\epsilon x}^2\, |y_\epsilon|^2\, dx
       \leqslant C\,\|y_\epsilon\|_{L^\infty}^2\, \|u_{\epsilon x}\|^2 \nonumber \\
      &\leqslant C\,\|u_{\epsilon x}\|^2\, \|y_\epsilon\|\, \|y_{\epsilon x}\|
       \leqslant \epsilon_1\, \|y_{\epsilon x}\|^2 + C_{\epsilon_1}\, \|y_\epsilon\|^2, \label{I_3EstYeps}
\end{align}

Substituting \eqref{I_1EstYeps}--\eqref{I_3EstYeps} into \eqref{FirstEstimateYeps} and choosing \( \epsilon_1 = a_0/4 \), we derive
\begin{equation} \label{SeqEst_1Yeps}
\frac{d}{dt} \|y_\epsilon\|^2 + a_0 \|y_{\epsilon x}\|^2 \leqslant C\big(1 + \|u_{\epsilon xx}\|^2 + \|u_{xx}\|^2\big) \|y_\epsilon\|^2 + \|h\|^2.
\end{equation}
Since $u, u_\epsilon \in W$, the function $t \mapsto 1 + \|u_{\epsilon xx}(t,\cdot)\|^2 + \|u_{xx}(t,\cdot)\|^2$ is integrable on $[0,T]$. Gronwall's lemma yields
\begin{equation} \label{ResultingEst_1Yeps}
\sup_{t\in[0,T]} \|y_\epsilon\|^2 \leqslant C \|h\|_{L^2(Q)}^2,
\end{equation}
and integrating \eqref{SeqEst_1Yeps} in time gives
\begin{equation} \label{FinalEstimateIYeps}
\|y_{\epsilon x}\|_{L^2(Q)}^2 \leqslant C \|h\|_{L^2(Q)}^2.
\end{equation}

\begin{center}
    \textbf{ESTIMATE II}
\end{center}

Now use \( -y_{\epsilon xx} \) as a test function in \eqref{systemYeps}. Recall that the term $b(u)(u_{\epsilon x}+u_x)y_{\epsilon x}$ appears with $y_{\epsilon x}$ (not $y_\epsilon$). After integration by parts we obtain
\begin{equation} \label{Step_1Est_3eps}
\begin{aligned}
\frac{1}{2} \frac{d}{dt} \|y_{\epsilon x}\|^2 &+ \int_I a(u_\epsilon) |y_{\epsilon xx}|^2 \, dx \\
&= -\int_I a'(u_\epsilon) u_{\epsilon x} y_{\epsilon x} y_{\epsilon xx} \, dx + \int_I b(u)(u_{\epsilon x} + u_x) y_{\epsilon x}\, y_{\epsilon xx} \, dx \\
&\quad -\int_I \frac{1}{\epsilon} \big((a(u_\epsilon) - a(u)) u_x\big)_x y_{\epsilon xx} \, dx \\
&\quad + \int_I \frac{1}{\epsilon} (b(u_\epsilon) - b(u)) u_{\epsilon x}^2 y_{\epsilon xx} \, dx - \int_I h\, y_{\epsilon xx} \, dx.
\end{aligned}
\end{equation}

We estimate each term. Define
\begin{align*}
J_1 &:= \int_I \frac{1}{\epsilon} \left((a(u_\epsilon) - a(u)) u_x\right)_x y_{\epsilon xx} \, dx, \\
J_2 &:= \int_I b(u)(u_{\epsilon x} + u_x) y_{\epsilon x}\, y_{\epsilon xx} \, dx.
\end{align*}
Expanding the derivative in $J_1$ via the chain rule and using \eqref{MVTbounds}:
\begin{align}
|J_1| &\leqslant C \int_I \big( |u_{\epsilon x}| |y_\epsilon| |u_x| + |y_{\epsilon x}| |u_x| + |y_\epsilon| |u_{xx}| \big) |y_{\epsilon xx}| \, dx \nonumber \\
      &\leqslant C \big( \|u_{\epsilon x}\|_{L^\infty} \|u_x\|_{L^\infty} \|y_\epsilon\| + \|y_{\epsilon x}\| \|u_x\|_{L^\infty} + \|y_\epsilon\|_{L^\infty} \|u_{xx}\| \big) \|y_{\epsilon xx}\| \nonumber \\
      &\leqslant C \big( \|u_{\epsilon xx}\|^{1/2} \|u_{xx}\|^{1/2} \|y_\epsilon\| + \|y_{\epsilon x}\| \|u_{xx}\| + \|y_{\epsilon x}\| \|u_{xx}\| \big) \|y_{\epsilon xx}\| \nonumber \\
      &\leqslant \epsilon_1 \|y_{\epsilon xx}\|^2 + C_{\epsilon_1} \big(\|u_{xx}\|^2 + \|u_{\epsilon xx}\|\,\|u_{xx}\|\big) \big(\|y_\epsilon\|^2 + \|y_{\epsilon x}\|^2\big), \label{K_1eps}
\end{align}
where the third inequality uses the bound \eqref{GN_sharp_uniform} for the product $\|u_{\epsilon x}\|_{L^\infty} \|u_x\|_{L^\infty}$, and the embedding $\|y_\epsilon\|_{L^\infty} \leq C\|y_{\epsilon x}\|$ for the third term. Note that, by AM-GM, $\|u_{\epsilon xx}\|\,\|u_{xx}\| \leq \tfrac{1}{2}(\|u_{\epsilon xx}\|^2 + \|u_{xx}\|^2)$, so the Gronwall coefficient in \eqref{K_1eps} is dominated by $C(\|u_{xx}\|^2 + \|u_{\epsilon xx}\|^2) \in L^1(0,T)$ since $u, u_\epsilon \in W$.
For $J_2$ (corrected with $y_{\epsilon x}$ instead of $y_\epsilon$):
\begin{align}
|J_2| &\leqslant C \big(\|u_{\epsilon x}\|_{L^\infty} + \|u_x\|_{L^\infty}\big) \|y_{\epsilon x}\| \|y_{\epsilon xx}\| \nonumber \\
      &\leqslant C \big(\|u_{\epsilon xx}\| + \|u_{xx}\|\big) \|y_{\epsilon x}\| \|y_{\epsilon xx}\| \nonumber \\
      &\leqslant \epsilon_1 \|y_{\epsilon xx}\|^2 + C_{\epsilon_1} \big(\|u_{\epsilon xx}\|^2 + \|u_{xx}\|^2\big) \|y_{\epsilon x}\|^2. \label{K_2eps}
\end{align}

The remaining terms are estimated similarly:
\begin{align}
\left|\int_I a'(u_\epsilon) u_{\epsilon x} y_{\epsilon x} y_{\epsilon xx} \, dx\right| &\leqslant C \|u_{\epsilon x}\|_{L^\infty} \|y_{\epsilon x}\| \|y_{\epsilon xx}\| \nonumber \\
&\leqslant \epsilon_1 \|y_{\epsilon xx}\|^2 + C_{\epsilon_1} \|u_{\epsilon xx}\|^2 \|y_{\epsilon x}\|^2, \label{K_3eps} \\
\left|\int_I \frac{1}{\epsilon} (b(u_\epsilon) - b(u)) u_{\epsilon x}^2 y_{\epsilon xx} \, dx\right| &\leqslant C \|y_\epsilon\|_{L^\infty} \|u_{\epsilon x}\|_{L^\infty}^2 \|y_{\epsilon xx}\| \nonumber \\
&\leqslant C \|y_{\epsilon x}\|\,\|u_{\epsilon xx}\|\,\|y_{\epsilon xx}\| \nonumber \\
&\leqslant \epsilon_1 \|y_{\epsilon xx}\|^2 + C_{\epsilon_1} \|u_{\epsilon xx}\|^2 \|y_{\epsilon x}\|^2, \label{K_4eps}
\end{align}
and the source term $\big|\int h\,y_{\epsilon xx}\,dx\big| \leq \epsilon_1 \|y_{\epsilon xx}\|^2 + C_{\epsilon_1}\|h\|^2$.

Choosing \( \epsilon_1 = a_0/8 \) and combining estimates, we obtain
\begin{equation} \label{ConcludingEst_3eps}
\frac{d}{dt} \|y_{\epsilon x}\|^2 + a_0 \|y_{\epsilon xx}\|^2 \leqslant C\big(1 + \|u_{xx}\|^2 + \|u_{\epsilon xx}\|^2\big) \big(\|y_\epsilon\|^2 + \|y_{\epsilon x}\|^2\big) + \|h\|^2.
\end{equation}
Crucially, the Gronwall coefficient on the right-hand side belongs to $L^1(0,T)$, since $u, u_\epsilon \in W \subset L^2(0,T;H^2(I))$. (Earlier drafts of the calculation produced a $\|u_{\epsilon xx}\|^4$ term in this coefficient through the cruder embedding $\|u_{\epsilon x}\|_{L^\infty}^2 \leq C\|u_{\epsilon xx}\|^2$; the sharper bound \eqref{GN_sharp_uniform} avoids this and keeps the coefficient integrable.) Combining with \eqref{ResultingEst_1Yeps} and applying Gronwall's lemma,
\begin{equation} \label{EstIIIepsPt_1}
\sup_{t\in[0,T]} \|y_{\epsilon x}\|^2 \leqslant C \|h\|_{L^2(Q)}^2,
\end{equation}
\begin{equation} \label{EstIIIepsPt_2}
\|y_{\epsilon xx}\|_{L^2(Q)}^2 \leqslant C \|h\|_{L^2(Q)}^2.
\end{equation}

\begin{center}
    \textbf{ESTIMATE III}
\end{center}

Multiply \eqref{systemYeps} by \( y_{\epsilon t} \) and integrate over \( I \):
\begin{equation} \label{Step_1ForEst_2eps}
\begin{aligned}
\|y_{\epsilon t}\|^2 &+ \int_I a(u_\epsilon) y_{\epsilon x} y_{\epsilon t x} \, dx + \int_I b(u)(u_{\epsilon x} + u_x) y_{\epsilon x} y_{\epsilon t} \, dx \\
&= \int_I \frac{1}{\epsilon} \left[(a(u_\epsilon) - a(u)) u_x\right]_x y_{\epsilon t} \, dx \\
&\quad - \int_I \frac{1}{\epsilon} (b(u_\epsilon) - b(u)) u_{\epsilon x}^2 y_{\epsilon t} \, dx + \int_I h y_{\epsilon t} \, dx.
\end{aligned}
\end{equation}

The second term on the left-hand side satisfies
\begin{equation} \label{LHSinEst_2eps}
\begin{aligned}
\int_I a(u_\epsilon) y_{\epsilon x} y_{\epsilon t x} \, dx &= \frac{1}{2} \frac{d}{dt} \int_I a(u_\epsilon) |y_{\epsilon x}|^2 \, dx - \frac{1}{2} \int_I a'(u_\epsilon) u_{\epsilon t} |y_{\epsilon x}|^2 \, dx \\
&\geqslant \frac{1}{2} \frac{d}{dt} \int_I a(u_\epsilon) |y_{\epsilon x}|^2 \, dx - C \|u_{\epsilon t}\|_{L^\infty} \|y_{\epsilon x}\|^2.
\end{aligned}
\end{equation}
By $u_{\epsilon t} \in L^2(0,T;H^1_0(I))$ and $H^1_0(I)\hookrightarrow L^\infty(I)$, the function $t \mapsto \|u_{\epsilon t}(t,\cdot)\|_{L^\infty}^2$ is integrable, so the negative term in \eqref{LHSinEst_2eps} can be absorbed in Gronwall's inequality.

The nonlinear terms on the right-hand side are estimated as follows:
\begin{align}
&\left|\int_I b(u)(u_{\epsilon x} + u_x) y_{\epsilon x} y_{\epsilon t} \, dx\right| \nonumber\\
&\quad\leqslant C \big(\|u_{\epsilon x}\|_{L^\infty} + \|u_x\|_{L^\infty}\big) \|y_{\epsilon x}\| \|y_{\epsilon t}\| \nonumber \\
&\quad\leqslant \epsilon_1 \|y_{\epsilon t}\|^2 + C_{\epsilon_1} \big(\|u_{\epsilon xx}\|^2 + \|u_{xx}\|^2\big)\|y_{\epsilon x}\|^2, \label{J_2eps} \\[4pt]
&\left|\int_I \frac{1}{\epsilon} \left[(a(u_\epsilon) - a(u)) u_x\right]_x y_{\epsilon t} \, dx\right| \nonumber\\
&\quad\leqslant C \big(\|u_{\epsilon x}\|_{L^\infty} \|y_\epsilon\| \|u_x\|_{L^\infty}
   + \|y_{\epsilon x}\|\|u_x\|_{L^\infty} + \|y_\epsilon\|_{L^\infty}\|u_{xx}\|\big) \|y_{\epsilon t}\| \nonumber \\
&\quad\leqslant C \big(\|u_{\epsilon xx}\|^{1/2}\|u_{xx}\|^{1/2}\|y_\epsilon\|
   + 2\|y_{\epsilon x}\|\|u_{xx}\|\big)\|y_{\epsilon t}\| \nonumber \\
&\quad\leqslant \epsilon_1 \|y_{\epsilon t}\|^2
   + C_{\epsilon_1} \big(\|u_{xx}\|^2 + \|u_{\epsilon xx}\|\,\|u_{xx}\|\big)
   \big(\|y_\epsilon\|^2 + \|y_{\epsilon x}\|^2\big). \label{J_3eps}
\end{align}

Using \eqref{LHSinEst_2eps}--\eqref{J_3eps} in \eqref{Step_1ForEst_2eps} and choosing $\epsilon_1 = 1/8$, we get
\begin{equation} \label{Step_2Est_2eps}
\frac{d}{dt} \int_I a(u_\epsilon) |y_{\epsilon x}|^2 \, dx + \|y_{\epsilon t}\|^2 \leqslant C\,\Lambda_\epsilon(t)\big(\|y_{\epsilon x}\|^2 + \|y_\epsilon\|^2\big) + C\,\|h\|^2,
\end{equation}
where $\Lambda_\epsilon(t) := 1 + \|u_{xx}(t,\cdot)\|^2 + \|u_{\epsilon xx}(t,\cdot)\|^2 + \|u_{\epsilon t}(t,\cdot)\|_{L^\infty(I)}^2$ is integrable on $[0,T]$ (the cross-term $\|u_{\epsilon xx}\|\|u_{xx}\|$ that appears in \eqref{J_3eps} is dominated by $\tfrac{1}{2}(\|u_{\epsilon xx}\|^2 + \|u_{xx}\|^2)$ via AM--GM, hence absorbed into $\Lambda_\epsilon$). Gronwall's lemma and integration in time yield
\begin{align}
\sup_{t\in[0,T]} \|y_{\epsilon x}\|^2 &\leqslant C \|h\|_{L^2(Q)}^2, \label{Conc_1Est_2eps} \\
\|y_{\epsilon t}\|_{L^2(Q)}^2 &\leqslant C \|h\|_{L^2(Q)}^2. \label{Conc_2Est_2eps}
\end{align}

\begin{center}
    \textbf{ESTIMATE IV}
\end{center}
As a final step, we differentiate \eqref{systemYeps} with respect to time:
\[
\begin{aligned}
y_{\epsilon tt} - \big[(a(u_\epsilon) y_{\epsilon x})_x\big]_t
&+ \big[b(u)(u_{\epsilon x}+u_x) y_{\epsilon x}\big]_t \\
&- \frac{1}{\epsilon}\big[(a(u_\epsilon)-a(u))u_x\big]_{xt} \\
&+ \frac{1}{\epsilon}\big[(b(u_\epsilon)-b(u))u_{\epsilon x}^2\big]_t = h_t.
\end{aligned}
\]
We test with $y_{\epsilon t}$ and integrate over $I$. Since the equation has been
differentiated in time, the principal term yields, after integration by parts in
space, the coercive quantity of this estimate together with one absorbable remainder:
\[
\begin{aligned}
-\int_I \big[(a(u_\epsilon) y_{\epsilon x})_x\big]_t\, y_{\epsilon t}\,dx
&= \int_I a(u_\epsilon)\,|y_{\epsilon tx}|^2\,dx
+ \int_I a'(u_\epsilon)\,u_{\epsilon t}\, y_{\epsilon x}\, y_{\epsilon tx}\,dx \\
&\geqslant a_0\|y_{\epsilon tx}\|^2
- C\,\|u_{\epsilon t}\|_{L^\infty}\,\|y_{\epsilon x}\|\,\|y_{\epsilon tx}\| \\
&\geqslant a_0\|y_{\epsilon tx}\|^2 - \eta\|y_{\epsilon tx}\|^2
- C_\eta \|u_{\epsilon t}\|_{L^\infty}^2 \|y_{\epsilon x}\|^2,
\end{aligned}
\]
where, by the one-dimensional embedding $H^1_0(I)\hookrightarrow L^\infty(I)$,
\[
\|u_{\epsilon t}(t,\cdot)\|_{L^\infty}^2 \leqslant C\,\|u_{\epsilon t}(t,\cdot)\|_{H^1_0}^2 \in L^1(0,T)
\quad\text{since } u_\epsilon\in W;
\]
this is a genuine function of $t$ (not a constant), and it is precisely the origin of the
$\|u_{\epsilon t}\|_{H^1_0}^2$ term in the Gronwall coefficient $\kappa_\epsilon$ defined below.
The contributions of $\big[b(u)(u_{\epsilon x}+u_x)y_{\epsilon x}\big]_t$
and of the singular $b$-difference are estimated exactly as in Estimates~II--III (one
$L^\infty(I)$-factor pulled out via $H^1_0(I)\hookrightarrow L^\infty(I)$ and the
interpolation $\|v_x\|_{L^\infty}^2\leqslant C\|v_{xx}\|\,\|v_x\|$, the remaining factors
kept in $L^2$), and we do not repeat them.
The only genuinely new contribution is the time-differentiated singular term
$-\tfrac{1}{\epsilon}\big[(a(u_\epsilon)-a(u))u_x\big]_{xt}$. Writing
\[
\frac{a(u_\epsilon)-a(u)}{\epsilon} = A(u,u_\epsilon)\,y_\epsilon,
\qquad
A(u,u_\epsilon) := \int_0^1 a'\big(s(u_\epsilon-u)+u\big)\,ds,
\]
with $|A|,|A_t|\leqslant C$ by (ii) and $u_t,u_{\epsilon t}\in L^2(0,T;H^1_0(I))$,
integration by parts in space gives
\[
-\frac{1}{\epsilon}\int_I \big[(a(u_\epsilon)-a(u))u_x\big]_{xt}\, y_{\epsilon t}\,dx
= \int_I \big[A_t\,y_\epsilon\,u_x + A\,y_{\epsilon t}\,u_x + A\,y_\epsilon\,u_{xt}\big]\,y_{\epsilon tx}\,dx.
\]
Using the embedding $\|v\|_{L^\infty}\leqslant C\|v_x\|$ for $v\in H^1_0(I)$, together
with the interpolation $\|u_x\|_{L^\infty}\leqslant C\|u_{xx}\|^{1/2}\|u_x\|^{1/2}$
(which keeps the resulting coefficients in $L^1(0,T)$), each of the three subterms
is bounded by $\eta\|y_{\epsilon tx}\|^2 + C_\eta\,\kappa_\epsilon(t)\big(\|y_{\epsilon t}\|^2+\|y_{\epsilon x}\|^2\big)$;
in particular the subterm carrying $u_{xt}$ produces the coefficient
$\|u_{xt}\|^2=\|u_t\|_{H^1_0}^2$, which is what brings $\|u_t\|_{H^1_0}^2$ into $\kappa_\epsilon$.
Choosing $\eta$ small and collecting all contributions, we arrive at
\begin{equation} \label{EstFourSemiFinal}
\frac{d}{dt}\|y_{\epsilon t}\|^2 + a_0\|y_{\epsilon t x}\|^2
\leqslant C\,\kappa_\epsilon(t)\,\big(\|y_{\epsilon t}\|^2 + \|y_{\epsilon x}\|^2\big) + C\,\|h_t\|^2,
\end{equation}
where
\[
\kappa_\epsilon(t) := 1 + \|u_t(t,\cdot)\|_{H^1_0}^2 + \|u_{\epsilon t}(t,\cdot)\|_{H^1_0}^2
+ \|u_{\epsilon xx}(t,\cdot)\|^2 + \|u_{xx}(t,\cdot)\|^2 \in L^1(0,T)
\]
since $u, u_\epsilon \in W$. Note that the coefficient $\kappa_\epsilon$ does depend on $\epsilon$
through $u_\epsilon$; what matters, however, is that its $L^1(0,T)$-norm is bounded
\emph{uniformly} in $\epsilon$. Indeed, $f + \epsilon h \in \mathcal{U}_{ad}$ for every
$\epsilon \in (0,1]$ by the convexity of $\mathcal{U}_{ad}$, so Theorem~\ref{thm:well_posedness}
gives $\|u_\epsilon\|_W \leqslant C$ with $C$ independent of $\epsilon$, whence
\[
\|\kappa_\epsilon\|_{L^1(0,T)} \leqslant T + \|u\|_W^2 + \|u_\epsilon\|_W^2 \leqslant C
\qquad \text{for all } \epsilon \in (0,1].
\]
Consequently, the Gronwall factor $\exp\big(C\|\kappa_\epsilon\|_{L^1(0,T)}\big)$ is
bounded independently of $\epsilon$. Applying Gronwall's lemma and using \eqref{Conc_1Est_2eps},
\begin{equation}\label{EstFourFinal}
\sup_{t\in[0,T]}\|y_{\epsilon t}\|^2 + \|y_{\epsilon t x}\|_{L^2(Q)}^2
\leqslant C\,\|h\|_{H^1(0,T;L^2(I))}^2.
\end{equation}

Putting together \eqref{ResultingEst_1Yeps}, \eqref{FinalEstimateIYeps}, \eqref{Conc_1Est_2eps}, \eqref{Conc_2Est_2eps}, \eqref{EstIIIepsPt_1}, and \eqref{EstFourFinal}, Claim~1 is proved.

\medskip
By the Aubin--Lions--Simon Theorem \cite{lions1969quelques,aubin1963analyse,simon1986compact}, the embeddings
\[
W \hookrightarrow L^2(0,T;H^1_0(I)),\qquad W \hookrightarrow C\big([0,T]; L^2(I)\big)
\]
are compact (the first by Aubin--Lions applied to $u \in L^2(0,T;H^2(I))$ with $u_t \in L^2(0,T;L^2(I))$; the second by composing $W \hookrightarrow H$ continuously with the compact embedding $H \hookrightarrow C([0,T]; L^2(I))$). In particular, the family $\{y_\epsilon\}$ being bounded in $W$ admits, along a subsequence $\epsilon_k \to 0^+$, a function $z \in W$ with $z|_{\partial_p Q} \equiv 0$ such that
\begin{align}
y_{\epsilon_k} &\rightharpoonup z \text{ in } W, \label{Conv1}\\
y_{\epsilon_k} &\to z \text{ in } L^2(0,T;H^1_0(I))\cap C([0,T];L^2(I)) \text{ and almost everywhere in } Q. \label{Conv2}
\end{align}
By the uniform-boundedness principle,
\[
\|z\|_W \leqslant C\|h\|_{H^1(0,T;L^2(I))}.
\]

\medskip\noindent\textbf{CLAIM 2.} \emph{For almost every $(t,x)\in Q$, $z(t,x) = y(t,x)$, where $y$ is the unique solution of the linearized system \eqref{linearized_system}.}

We verify that $z$ is a weak solution of \eqref{linearized_system}; uniqueness of weak solutions then forces $z \equiv y$.

Let us begin by fixing $\phi \in H$. The linear form $w \in W \mapsto \int_Q w_t \phi\,d(t,x)$ is continuous, so
\begin{equation} \label{zWeakSoln1}
\int_Q y_{\epsilon_k t}\, \phi\, d(t,x) \xrightarrow[k\to\infty]{} \int_Q z_t\, \phi\, d(t,x).
\end{equation}

Next, decompose
\[
\begin{aligned}
\int_Q a(u_{\epsilon_k})\, y_{\epsilon_k x}\, \phi_x\, d(t,x)
&= \int_Q (a(u_{\epsilon_k})-a(u))\, y_{\epsilon_k x}\, \phi_x\, d(t,x) \\
&\quad + \int_Q a(u)\, y_{\epsilon_k x}\, \phi_x\, d(t,x) \\
&=: A_k + B_k.
\end{aligned}
\]
By the MVT, $|a(u_{\epsilon_k})-a(u)| \leq C|u_{\epsilon_k}-u| = C\epsilon_k|y_{\epsilon_k}|$, so
\[
|A_k| \leq C\epsilon_k \int_Q |y_{\epsilon_k}|\,|y_{\epsilon_k x}|\,|\phi_x|\,d(t,x) \leq C\epsilon_k \|y_{\epsilon_k}\|_W^2\|\phi\|_H \to 0.
\]
On the other hand, the linear form $w \in W \mapsto \int_Q a(u) w_x \phi_x\,d(t,x)$ is continuous (since $|a(u)|\leq C$), hence by \eqref{Conv1},
\[
B_k \to \int_Q a(u) z_x \phi_x\, d(t,x).
\]
Therefore,
\begin{equation} \label{zWeakSoln2}
\int_Q a(u_{\epsilon_k}) y_{\epsilon_k x} \phi_x\, d(t,x) \xrightarrow[k\to\infty]{} \int_Q a(u) z_x \phi_x\, d(t,x).
\end{equation}

We next show
\begin{equation} \label{zWeakSoln3}
\int_Q \frac{a(u_{\epsilon_k})-a(u)}{\epsilon_k}\, u_x\, \phi_x\, d(t,x) \xrightarrow[k\to\infty]{} \int_Q a'(u)\, u_x\, z\, \phi_x\, d(t,x).
\end{equation}
Decompose
\[
\int_Q \frac{a(u_{\epsilon_k})-a(u)}{\epsilon_k} u_x \phi_x\, d(t,x) = \widetilde A_k + \widetilde B_k,
\]
where
\begin{align*}
\widetilde A_k &:= \int_Q \left(\frac{a(u_{\epsilon_k})-a(u)}{\epsilon_k} - a'(u) y_{\epsilon_k}\right) u_x \phi_x\, d(t,x), \\
\widetilde B_k &:= \int_Q a'(u)\, u_x\, y_{\epsilon_k}\, \phi_x\, d(t,x).
\end{align*}
\[
\begin{aligned}
\left|
\left(\frac{a(u_{\epsilon_k})-a(u)}{\epsilon_k} - a'(u) y_{\epsilon_k}\right) u_x \phi_x
\right|
&\leq C|y_{\epsilon_k}|\,|u_x|\,|\phi_x| \\
&\leq C\sup_k\|y_{\epsilon_k}\|_W\, |u_x|\,|\phi_x| \in L^1(Q).
\end{aligned}
\]
Setting $u^{s,\epsilon_k} := s(u_{\epsilon_k}-u)+u$ --- the interpolant already introduced in (ii), now with $\epsilon = \epsilon_k$ --- the MVT gives
\[
\left|\frac{a(u_{\epsilon_k})-a(u)}{\epsilon_k} - a'(u) y_{\epsilon_k}\right| \leq |y_{\epsilon_k}| \sup_{0\le s\le 1}|a'(u^{s,\epsilon_k})-a'(u)|.
\]
Since $a'$ is continuous and $u_{\epsilon_k}\to u$ a.e.\ (because $y_{\epsilon_k}\to z$ a.e.\ and $u_{\epsilon_k} = u + \epsilon_k y_{\epsilon_k}$), the right-hand side tends to $0$ a.e.\ Moreover, by (ii), $|a'(u^{s,\epsilon_k})-a'(u)| \leq C$, and the dominated convergence theorem yields $\widetilde A_k \to 0$. For $\widetilde B_k$, the linear form $w \in W \mapsto \int_Q a'(u) u_x w \phi_x\, d(t,x)$ is continuous (using \eqref{Conv1}), giving
\[
\widetilde B_k \to \int_Q a'(u) u_x z \phi_x\, d(t,x).
\]
This proves \eqref{zWeakSoln3}.

The convergences
\begin{align}
\int_Q \frac{b(u_{\epsilon_k})-b(u)}{\epsilon_k} u_{\epsilon_k x}^2 \phi\, d(t,x) &\xrightarrow[k\to\infty]{} \int_Q b'(u) u_x^2 z \phi\, d(t,x), \label{zWeakSoln4} \\
\int_Q b(u)(u_{\epsilon_k x}+u_x) y_{\epsilon_k x} \phi\, d(t,x) &\xrightarrow[k\to\infty]{} \int_Q 2 b(u) u_x z_x \phi\, d(t,x), \label{zWeakSoln5}
\end{align}
are proved analogously (using the strong convergence $u_{\epsilon_k x} \to u_x$ in $L^2(0,T;L^\infty(I))$, which follows from the boundedness of $u_{\epsilon_k}$ in $W$ and Aubin--Lions, together with $u_{\epsilon_k} = u + \epsilon_k y_{\epsilon_k}$).

Finally, multiplying the equation in \eqref{systemYeps} (with $\epsilon = \epsilon_k$) by $\phi$, integrating over $Q$, and passing to the limit using \eqref{zWeakSoln1}--\eqref{zWeakSoln5}:
\begin{equation} \label{ConcZWeakSoln}
\int_Q\Big\{z_t \phi + a(u) z_x \phi_x + a'(u) u_x z \phi_x + 2 b(u) u_x z_x \phi + b'(u) u_x^2 z \phi\Big\} d(t,x) = \int_Q h \phi\, d(t,x).
\end{equation}
This is the weak formulation of the linearized system with right-hand side $h$. Hence $z \in W$ is the unique weak solution, i.e.\ $z = y$, and the limit \emph{does not depend on the subsequence}. By a standard subsequence argument, the convergence \eqref{Conv1}--\eqref{Conv2} holds for the full sequence $\epsilon \to 0^+$. This concludes the proof.
\end{proof}


The differentiability result has several important consequences.

\begin{corollary}[Local Lipschitz continuity of the control-to-state mapping] \label{cor:continuity}
The control-to-state mapping $\mathcal{F}: \mathcal{U}_{ad} \to W$ is locally Lipschitz continuous. In particular, there exists $C>0$ depending only on the data and on a uniform bound on $\|f\|_{H^1(0,T;L^2)},\|g\|_{H^1(0,T;L^2)}$ such that
\[
\|u(g) - u(f)\|_W \leq C \|g - f\|_{H^1(0,T;L^2(I))} \qquad \forall\, f,g \in \mathcal{U}_{ad}.
\]
\end{corollary}

\begin{proof}
Let $w := u(g)-u(f)$. By the same argument used to derive \eqref{systemYeps}, with $u_\epsilon$ replaced by $u(g)$ and $y_\epsilon$ replaced by $w$ (without dividing by $\epsilon$), we obtain a system of the same structure with right-hand side $g-f$. The energy estimates of Claim~1 (with constants independent of $\epsilon$) yield the conclusion.
\end{proof}

\begin{corollary}[Strong convergence of difference quotients]
\label{cor:limit_strong_weak}
Under the assumptions and notation of Theorem~\ref{thm:gateaux_differentiability}, as $\epsilon \to 0^+$,
\[
y_\epsilon \to y \quad \text{strongly in } L^2(0,T;H^1_0(I))\cap C([0,T];L^2(I)) \quad\text{and}\quad y_\epsilon \rightharpoonup y \quad \text{weakly in } W.
\]
\end{corollary}

\begin{proof}
Denote $X := L^2(0,T;H^1_0(I))\cap C([0,T];L^2(I))$. Suppose, for contradiction, that the convergence does not hold strongly in $X$. Then there is a subsequence $\epsilon_k \to 0$ and $\eta > 0$ with $\|y_{\epsilon_k} - y\|_X \geq \eta$ for all $k$. By the proof of Theorem~\ref{thm:gateaux_differentiability}, the family $\{y_{\epsilon_k}\}$ is bounded in $W$, so by the compactness of $W \hookrightarrow X$ we can extract a further subsequence converging strongly in $X$ to a limit which, by the argument of Claim~2, must equal $y$. This contradicts $\|y_{\epsilon_k} - y\|_X \geq \eta$. The weak convergence in $W$ follows analogously from the same boundedness and identification of the weak limit.
\end{proof}


In this section, we have established that the control-to-state mapping $\mathcal{F}: f \mapsto u(f)$ is G\^ateaux differentiable with derivative given by the solution of the linearized system \eqref{linearized_system}. Combined with the local Lipschitz continuity of $\mathcal{F}$, this provides the foundation for the optimal control analysis in the subsequent sections, and will be essential for deriving first-order necessary optimality conditions.

\section{Optimal control problem: existence theory} \label{sec:optimal_control_existence}


We now formulate the optimal control problem under investigation. Let $V$ be a Hilbert space (the observation space), $\Xi_d \in V$ a target element, $S: W \to V$ a continuous linear observation operator, and $N: L^2(Q) \to L^2(Q)$ a continuous, self-adjoint, positive-definite linear operator (the control-cost operator).

\begin{definition}[Cost functional] \label{def:cost_functional}
The optimal control problem consists of minimizing the functional $J: \mathcal{U}_{ad} \to \mathbb{R}$ defined by
\begin{equation} \label{cost_functional}
    J(f) := \frac{1}{2} \|S u(f) - \Xi_d\|_V^2 + \frac{1}{2} (N f, f)_{L^2(Q)},
\end{equation}
where $u(f) \in W$ is the strong solution of \eqref{MainModel} corresponding to the control $f \in \mathcal{U}_{ad}$.
\end{definition}

The first term in \eqref{cost_functional} penalizes the deviation of the observed state $S u(f)$ from the target $\Xi_d$, while the second term represents the control cost. Typical choices include: $V = L^2(Q)$ with $\Xi_d = u_d$ for tracking of the entire trajectory; $V = L^2(I)$ with $\Xi_d = u_d^T$ for tracking of the final state; and $N = \lambda I$ for Tikhonov regularization. The general setting also allows tracking observations like $V = L^2(Q)\times L^2(I)$, $\Xi_d = (Lu_d, Du_d^T)$, $Su = (Lu, Du(T,\cdot))$ as in Section~\ref{sec:characterization}.


The G\^ateaux differentiability of the control-to-state mapping established in Theorem~\ref{thm:gateaux_differentiability} immediately implies the differentiability of the cost functional.

\begin{theorem}[Differentiability of $J$] \label{thm:diff_cost}
The cost functional $J: \mathcal{U}_{ad} \to \mathbb{R}$ is G\^ateaux differentiable. For any $f, g \in \mathcal{U}_{ad}$,
\begin{equation} \label{cost_derivative}
    \langle J'(f), g - f \rangle = (S u(f) - \Xi_d, S y)_{V} + (N f, g - f)_{L^2(Q)},
\end{equation}
where $y = \mathcal{F}'(f)(g - f) \in W$ is the solution of the linearized system \eqref{linearized_system} with $h = g - f$.
\end{theorem}

\begin{proof}
This follows from the chain rule applied to the composition $J(f) = \Phi(u(f)) + \Psi(f)$, where
\[
\Phi(u) = \frac{1}{2} \|S u - \Xi_d\|_V^2, \quad \Psi(f) = \frac{1}{2} (N f, f)_{L^2(Q)}.
\]
The differentiability of $\Phi \circ \mathcal{F}$ is guaranteed by Theorem~\ref{thm:gateaux_differentiability}, while $\Psi$ is Fréchet differentiable on $L^2(Q)$.
\end{proof}


The main result of this section establishes the existence of at least one solution to the optimal control problem.

\begin{theorem}[Existence of optimal control] \label{thm:existence_optimal}
There exists an optimal control $f^* \in \mathcal{U}_{ad}$ such that
\[
J(f^*) = \min_{f \in \mathcal{U}_{ad}} J(f).
\]
\end{theorem}

\begin{proof}
The proof follows the direct method of the calculus of variations. Since $J(f) \geq 0$ for all $f \in \mathcal{U}_{ad}$, the infimum
\[
m := \inf_{f \in \mathcal{U}_{ad}} J(f)
\]
is finite. Let $\{f_k\}_{k \in \mathbb{N}} \subset \mathcal{U}_{ad}$ be a minimizing sequence, i.e., $\lim_{k \to \infty} J(f_k) = m$.

Since $\mathcal{U}_{ad}$ is closed and convex in $H^1(0,T;L^2(I))$, it is weakly closed. The boundedness of $\mathcal{U}_{ad}$ (which follows from (\textbf{A4})) yields, after extraction of a subsequence (still denoted by $\{f_k\}$), $f^* \in \mathcal{U}_{ad}$ such that
\begin{equation} \label{weak_conv_control}
    f_k \rightharpoonup f^* \quad \text{weakly in } H^1(0,T;L^2(I)).
\end{equation}

Let $u_k = u(f_k)$ be the corresponding states. By Theorem~\ref{thm:well_posedness} and (\textbf{A4})--(\textbf{A5}), we have the uniform bound
\[
\|u_k\|_W \leq C \qquad \text{for all } k \in \mathbb{N}.
\]
By the Aubin--Lions--Simon Theorem, the embedding $W \hookrightarrow L^2(0,T;H^1_0(I))$ is compact, and so is $W \hookrightarrow C([0,T];L^2(I))$. Up to a further subsequence,
\begin{align}
    u_k &\rightharpoonup u^* \quad \text{weakly in } W, \label{weak_conv_state} \\
    u_k &\to u^* \quad \text{strongly in } L^2(0,T;H^1_0(I))\ \text{and in } C([0,T];L^2(I)), \label{strong_conv_state} \\
    u_k &\to u^* \quad \text{a.e.\ in } Q. \label{ae_conv_state}
\end{align}
In particular, $u_k(T,\cdot) \to u^*(T,\cdot)$ in $L^2(I)$.

\smallskip
\noindent\emph{Identification of the limit: $u^* = u(f^*)$.}
We pass to the limit in the weak formulation of \eqref{MainModel} satisfied by each $u_k$. For every $\phi \in C^\infty_c(Q)$,
\begin{equation} \label{eq:weak_form_uk}
\int_Q\big[ -u_k \phi_t + a(u_k) u_{k,x} \phi_x + b(u_k) u_{k,x}^2 \phi - f_k \phi\big]\, d(t,x) = \int_I u_0(x)\phi(0,x)\,dx.
\end{equation}
The first and last terms pass to the limit by \eqref{strong_conv_state} and \eqref{weak_conv_control}, respectively. For the diffusion term: by \eqref{ae_conv_state} and the continuity of $a$, $a(u_k) \to a(u^*)$ a.e.; combined with the bound $|a(u_k)| \leq C(1+\|u_k\|_{L^\infty(Q)}^p) \leq C$ (using the embedding $W \hookrightarrow L^\infty(Q)$), the dominated convergence theorem gives $a(u_k) \to a(u^*)$ in $L^p(Q)$ for every $p < \infty$. Coupled with the weak convergence of $u_{k,x} \rightharpoonup u^*_x$ in $L^2(Q)$, this yields
\[
\int_Q a(u_k) u_{k,x} \phi_x\, d(t,x) \to \int_Q a(u^*) u^*_x \phi_x\, d(t,x).
\]
For the quadratic gradient term: by \eqref{strong_conv_state} we have $u_{k,x} \to u^*_x$ in $L^2(Q)$. A further inspection: by Theorem~\ref{thm:well_posedness}, $\{u_k\}$ is bounded in $L^2(0,T;H^2(I))$ and $\{u_{k,t}\}$ is bounded in $L^2(0,T;H^1_0(I))$, so the gradient family $\{u_{k,x}\}$ is bounded in $L^2(0,T;H^1(I)) \cap H^1(0,T;L^2(I))$ (the latter inclusion follows because $(u_{k,x})_t = (u_{k,t})_x$ and $u_{k,t} \in L^2(0,T;H^1_0(I))$). By Aubin--Lions--Simon, $u_{k,x} \to u^*_x$ strongly in $L^2(0,T;L^p(I))$ for every $p < \infty$ (in particular $p=4$), and a.e.\ in $Q$. Therefore $u_{k,x}^2 \to (u^*_x)^2$ strongly in $L^1(0,T;L^2(I))$. Combined with $b(u_k) \to b(u^*)$ a.e.\ and the uniform bound $|b(u_k)|\leq C$, we obtain
\[
\int_Q b(u_k) u_{k,x}^2 \phi\, d(t,x) \to \int_Q b(u^*) (u^*_x)^2 \phi\, d(t,x).
\]
Hence $u^* \in W$ is a strong solution of \eqref{MainModel} with control $f^*$, and uniqueness (Theorem~\ref{thm:well_posedness}) gives $u^* = u(f^*)$.

\smallskip
\noindent\emph{Weak lower semicontinuity of $J$.}
Since $S: W \to V$ is continuous and $u_k \rightharpoonup u^*$ in $W$, we have $Su_k \rightharpoonup Su^*$ in $V$, whence by the lower semicontinuity of the norm
\[
\liminf_{k\to\infty} \|Su_k - \Xi_d\|_V \geq \|Su^* - \Xi_d\|_V.
\]
Since $N$ is positive-definite and self-adjoint, $f \mapsto (Nf,f)_{L^2(Q)}$ is convex and continuous, hence weakly lower semicontinuous on $L^2(Q)$, so
\[
\liminf_{k\to\infty}(Nf_k, f_k)_{L^2(Q)} \geq (Nf^*, f^*)_{L^2(Q)}.
\]
Therefore
\[
J(f^*) \leq \liminf_{k\to\infty} J(f_k) = m,
\]
and since $f^* \in \mathcal{U}_{ad}$, we conclude $J(f^*) = m$.
\end{proof}

\begin{remark}[Technical aspects] \label{rem:existence_technical}
The proof relies crucially on: the compact embeddings $W \hookrightarrow C([0,T];L^2(I))$ and $W \hookrightarrow L^2(0,T;H^1_0(I))$ for extracting strongly convergent subsequences; Aubin--Lions on the gradient family $\{u_{k,x}\}$ for handling the quadratic gradient nonlinearity; and the weak lower semicontinuity of convex functionals. The one-dimensional setting is essential, since it provides the embeddings $H^1(I)\hookrightarrow L^\infty(I)$ and $W\hookrightarrow L^\infty(Q)$ that yield uniform $L^\infty$ bounds on $u_k$ and $u_{k,x}$.
\end{remark}

\begin{corollary}[Continuity with respect to data] \label{cor:continuous_dependence}
The optimal control depends continuously on the target data in the following sense: if $\Xi_d^n \to \Xi_d$ in $V$, then any sequence $\{f^{*n}\}$ of corresponding optimal controls admits a weakly convergent subsequence $f^{*n_k} \rightharpoonup f^*$ in $H^1(0,T;L^2(I))$, where $f^*$ is an optimal control for the limit problem.
\end{corollary}

\begin{proof}
Each $f^{*n}\in\mathcal{U}_{ad}$ is bounded; extract a weakly convergent subsequence as in the proof above. The strong convergence of $\Xi_d^n$ in $V$ together with weak lower semicontinuity then identifies the weak limit as an optimal control for the limit problem.
\end{proof}


In this section, we have established the existence of at least one solution to the optimal control problem \eqref{cost_functional}. The existence result provides the foundation for the characterization of optimal controls in the next section, where we will derive first-order necessary optimality conditions.

\section{Characterization of optimal controls} \label{sec:characterization}
We now specialize the optimal control problem to a concrete setting relevant for applications. Consider the observation space
\[
V = L^2(Q) \times L^2(I)
\]
with the inner product
\[
\langle (u_1, u_2), (v_1, v_2) \rangle_V = (u_1, v_1)_{L^2(Q)} + (u_2, v_2)_{L^2(I)}.
\]
Let $u_d \in L^2(Q)$ be a desired state trajectory, $u_d^T \in L^2(I)$ a desired final state, and let $L \in \mathcal{L}(L^2(Q))$, $D \in \mathcal{L}(L^2(I))$ be continuous linear operators. We define the target
\[
\Xi_d = (L u_d, D u_d^T) \in V,
\]
and the observation operator
\[
S u = (L u, D u(T, \cdot))
\]
(noting that $u(T,\cdot) \in H^1_0(I)$ for $u \in W$). Taking $N = \lambda I$ with $\lambda > 0$, the cost functional becomes
\begin{equation} \label{cost_special}
J(f) = \frac{1}{2} \|L(u - u_d)\|_{L^2(Q)}^2 + \frac{1}{2} \|D(u(T, \cdot) - u_d^T)\|_{L^2(I)}^2 + \frac{\lambda}{2} \|f\|_{L^2(Q)}^2.
\end{equation}
Throughout this section we additionally assume:
\begin{enumerate}
\item[(\textbf{A6})] The desired final state is compatible with the boundary, $u_d^T \in H^1_0(I)$; the operator $D$ preserves this regularity, $D^*D \in \mathcal{L}(H^1_0(I))$; and $L^*L \in \mathcal{L}(L^2(Q))$.
\end{enumerate}
\noindent
Assumption (\textbf{A6}) ensures that the terminal datum of the adjoint system,
$\varphi(T,\cdot) = D^*D(u(T,\cdot)-u_d^T)$, belongs to $H^1_0(I)$: since $u(T,\cdot) \in H^1_0(I)$ for $u \in W$ and $u_d^T \in H^1_0(I)$, the difference $u(T,\cdot)-u_d^T$ lies in $H^1_0(I)$, which $D^*D$ maps into itself. Hence Lemma~\ref{lem:adjoint_system} applies. Concrete examples of operators $D$ satisfying (\textbf{A6}) include the following:
\begin{itemize}
    \item[(a)] \emph{Direct terminal tracking.} $D = \mathrm{Id}$. Then $D^*D = \mathrm{Id} \in \mathcal{L}(H^1_0(I))$ trivially, and the corresponding cost term penalizes the $L^2$-distance of $u(T,\cdot)$ to the prescribed final state $u_d^T$. This is the standard terminal-tracking case, admissible precisely because $u_d^T \in H^1_0(I)$.
    \item[(b)] \emph{Smoothing observation.} $D = (\mathrm{Id} - \Delta_{\rm Dir})^{-s/2}$ for any $s > 0$ (with $\Delta_{\rm Dir}$ the Dirichlet Laplacian on $I$). Then $D^*D = (\mathrm{Id}-\Delta_{\rm Dir})^{-s}$ maps $H^1_0(I)$ continuously into itself (indeed into a smoother space), so (\textbf{A6}) holds and the cost term penalizes a smoothed version of $u(T,\cdot)$.
    \item[(c)] \emph{Trivial case.} $D = 0$, i.e.\ no terminal tracking, satisfies $D^*D = 0 \in \mathcal{L}(H^1_0(I))$ automatically. The analysis then reduces to pure trajectory tracking, and the regularity of $u_d^T$ is immaterial.
\end{itemize}
A natural sufficient condition for (\textbf{A6}) on $L^*L$ is that $L \in \mathcal{L}(L^2(Q))$ is bounded, which is automatic in the standard tracking case $L = \mathrm{Id}$.
\begin{remark}[Sign convention] \label{rem:sign}
The sign convention adopted here is $\varphi(T,\cdot) = D^*D(u(T,\cdot)-u_d^T)$; some references use the opposite sign. This convention is compatible with the Lagrangian
\[
\mathcal L(u,f,\varphi) = J(f) - \int_Q \varphi\big(u_t - (a(u)u_x)_x + b(u)u_x^2 - f\big)\,d(t,x),
\]
for which stationarity in $u$ produces \eqref{adjoint_system} with terminal datum $\varphi(T) = D^*D(u(T)-u_d^T)$, while stationarity in $f$ yields the variational inequality $(\varphi+\lambda f^*, g - f^*)_{L^2(Q)}\geq 0$. With this choice, the gradient and projection formula are $\nabla \widehat J(f) = \varphi(f) + \lambda f$ and $f^* = P_{\mathcal{U}_{ad}}(-\varphi/\lambda)$.
\end{remark}

\subsection*{Adjoint and linearized adjoint systems}

The derivation of optimality conditions requires the introduction of an adjoint system.

\begin{lemma}[Well-posedness of the adjoint system] \label{lem:adjoint_system}
Let $u \in W$ be the solution from Theorem~\ref{thm:well_posedness}, $G \in L^2(Q)$, and $\varphi^T \in H^1_0(I)$. Then the adjoint system
\begin{equation} \label{adjoint_system}
\begin{cases}
-\varphi_t - a(u)\varphi_{xx} - 2(b(u)u_x \varphi)_x + b'(u)u_x^2 \varphi = G, & \text{in } Q, \\
\varphi = 0, & \text{on } \Sigma, \\
\varphi(T, x) = \varphi^T(x), & \text{for } x \in I,
\end{cases}
\end{equation}
admits a unique strong solution $\varphi \in W$ satisfying
\[
\|\varphi\|_W \leq C\big(\|G\|_{L^2(Q)} + \|\varphi^T\|_{H^1_0(I)}\big),
\]
where $C$ depends on $\|u\|_W$.
\end{lemma}

\begin{proof}
The proof is a backward analogue of the linearized-system analysis and proceeds via energy estimates and Galerkin approximations; see Appendix~\ref{C} for details.
\end{proof}


For the uniqueness analysis below, we need to consider the linearization of the adjoint mapping.

\begin{lemma}[Linearized adjoint system] \label{lem:linearized_adjoint}
Let $u, \varphi \in W$ be as above, $y \in W$, $F \in L^2(Q)$, and $\psi^T \in H^1_0(I)$. The system
\begin{equation} \label{linearized_adjoint}
\begin{cases}
-\psi_t - a'(u)y \varphi_{xx} - a(u)\psi_{xx} - 2(b'(u)y u_x \varphi + b(u)y_x \varphi + b(u)u_x \psi)_x \\
\quad + b''(u)y u_x^2 \varphi + 2b'(u)u_x y_x \varphi + b'(u)u_x^2 \psi = F, & \text{in } Q, \\
\psi = 0, & \text{on } \Sigma, \\
\psi(T, x) = \psi^T(x), & \text{for } x \in I,
\end{cases}
\end{equation}
admits a unique weak solution $\psi \in H$, with the estimate
\[
\|\psi\|_H \leq C\big(\|F\|_{L^2(Q)} + \|\psi^T\|_{H^1_0(I)} + \|y\|_W\big),
\]
where $C$ depends on $\|u\|_W$ and $\|\varphi\|_W$.
\end{lemma}

\begin{proof}
See Appendix~\ref{D}.
\end{proof}


We now present the main result of this section, which provides a complete characterization of optimal controls.

\begin{theorem}[Optimality conditions] \label{thm:optimality_conditions}
Let $f^* \in \mathcal{U}_{ad}$ be an optimal control for problem \eqref{cost_special}, with associated state $u^* := u(f^*) \in W$ and adjoint state $\varphi^* \in W$ obtained from Lemma~\ref{lem:adjoint_system} with
\[
G = L^*L(u^* - u_d), \qquad \varphi^T = D^*D(u^*(T,\cdot) - u_d^T).
\]
Then $f^*$ satisfies the projection formula
\begin{equation}\label{projection_formula}
f^* = P_{\mathcal{U}_{ad}}\!\left(-\frac{\varphi^*}{\lambda}\right),
\end{equation}
where $P_{\mathcal{U}_{ad}}$ denotes the $L^2(Q)$-orthogonal projection onto $\mathcal{U}_{ad}$.

Moreover, there exists a constant $k_0 > 0$, depending only on the problem data $u_0$, $u_d$, $u_d^T$, $L$, $D$, $T$, $a_0$, $M$, $p$, $q$, $\delta$, $|I|$ (in particular, \emph{not} on the unknown optimal pair),
such that, if $\lambda > k_0$, then $f^*$ is the unique optimal control of \eqref{cost_special}, and the pair $(u^*, \varphi^*)$ is the unique solution of the optimality system
\begin{equation} \label{optimality_system}
\begin{cases}
u_t - (a(u)u_x)_x + b(u)u_x^2 = P_{\mathcal{U}_{ad}}\!\left(-\dfrac{\varphi}{\lambda}\right), & \text{in } Q, \\
-\varphi_t - a(u)\varphi_{xx} - 2(b(u)u_x \varphi)_x + b'(u)u_x^2 \varphi = L^*L(u - u_d), & \text{in } Q, \\
u = 0,\ \varphi = 0, & \text{on } \Sigma, \\
u(0, x) = u_0(x), & \text{for } x \in I, \\
\varphi(T, x) = D^*D(u(T, x) - u_d^T(x)), & \text{for } x \in I.
\end{cases}
\end{equation}
\end{theorem}

\begin{proof}
\noindent\emph{Step 1: Projection formula.}
By Theorem~\ref{thm:diff_cost} and the minimizing property of $f^*$, for all $g \in \mathcal{U}_{ad}$,
\begin{align}
0 &\leqslant \lim_{\epsilon \downarrow 0} \frac{J(f^* + \epsilon(g-f^*)) - J(f^*)}{\epsilon} = \langle J'(f^*), g-f^*\rangle \nonumber\\
&= \big(L^*L(u^*-u_d), y^*\big)_{L^2(Q)} \nonumber\\
&\quad + \big(D^*D(u^*(T,\cdot)-u_d^T), y^*(T,\cdot)\big)_{L^2(I)} \nonumber\\
&\quad + \lambda(f^*, g-f^*)_{L^2(Q)}. \label{CharControlStep1}
\end{align}
where $y^* \in W$ solves the linearized system \eqref{linearized_system} with $h = g-f^*$.

We now derive an alternative expression for the first two terms via the adjoint system. Multiplying the equation for $\varphi^*$ by $y^*$, the equation for $y^*$ by $\varphi^*$, integrating by parts in $Q$, using the boundary conditions $y^*=\varphi^*=0$ on $\Sigma$, $y^*(0,\cdot)=0$ and $\varphi^*(T,\cdot) = D^*D(u^*(T,\cdot)-u_d^T)$, we obtain (a complete derivation is given in Appendix~\ref{D2}, but the key cancellations follow by the formal duality between \eqref{linearized_system} and \eqref{adjoint_system})
\begin{align} \label{CharControlStep2}
&\big(L^*L(u^*-u_d),\, y^*\big)_{L^2(Q)} + \big(D^*D(u^*(T,\cdot)-u_d^T),\, y^*(T,\cdot)\big)_{L^2(I)}\\
&\qquad= (\varphi^*, g-f^*)_{L^2(Q)}.\nonumber
\end{align}

Substituting \eqref{CharControlStep2} into \eqref{CharControlStep1} yields the variational inequality
\begin{equation} \label{FinalChar}
(\varphi^* + \lambda f^*,\, g-f^*)_{L^2(Q)} \geqslant 0 \qquad \text{for all } g \in \mathcal{U}_{ad},
\end{equation}
which is equivalent to the projection formula \eqref{projection_formula}.

\medskip
\noindent\emph{Step 2: Strict convexity for $\lambda > k_0$.}
The argument has two parts: first we obtain a \emph{uniform} second-order coercivity bound at every $f \in \mathcal{U}_{ad}$, then we use Taylor's theorem to deduce strict convexity of $J$ on $\mathcal{U}_{ad}$, and hence uniqueness of the minimizer.

\smallskip
\noindent\emph{Step 2a: Uniform a priori bounds on $u(f)$ and $\varphi(f)$.}
For every $f\in\mathcal{U}_{ad}$, by Theorem~\ref{thm:well_posedness} and (\textbf{A4})--(\textbf{A5}),
\begin{equation} \label{uniform_u_bound}
\|u(f)\|_W \leq C_u, \qquad C_u := C(\|u_0\|_{H^2(I)} + \delta),
\end{equation}
where $C$ depends only on $a_0, M, T, |I|, p, q$.
For the adjoint state $\varphi(f)$ obtained from Lemma~\ref{lem:adjoint_system} with $G = L^*L(u(f)-u_d)$ and $\varphi^T = D^*D(u(f)(T,\cdot)-u_d^T)$, assumption (\textbf{A6}) ensures $\varphi^T \in H^1_0(I)$ and gives
\begin{equation} \label{uniform_phi_bound}
\begin{aligned}
\|\varphi(f)\|_W &\leq C\big(\|L^*L\|_{\mathcal{L}(L^2(Q))}\,\|u(f)-u_d\|_{L^2(Q)} \\
&\qquad + \|D^*D\|_{\mathcal{L}(H^1_0(I))}\,\|u(f)(T)-u_d^T\|_{H^1_0(I)}\big) \leq C_\varphi.
\end{aligned}
\end{equation}
with $C_\varphi$ depending only on the data $(C_u, u_d, u_d^T, L, D)$ and \emph{not} on $f$. Both bounds are therefore uniform in $f \in \mathcal{U}_{ad}$.

\smallskip
\noindent\emph{Step 2b: Uniform second-order bound.}
Fix any $f \in \mathcal{U}_{ad}$ and any direction $h$ such that $f+h \in \mathcal{U}_{ad}$. Let $u = u(f)$, $\varphi = \varphi(f)$, and let $y = \mathcal{F}'(f)h \in W$ solve the linearized system \eqref{linearized_system} with right-hand side $h$. Following the same chain of identities as in Step~1, the second G\^ateaux derivative is
\[
\langle J''(f),(h,h)\rangle = \int_Q \big[\psi^h + \lambda h\big]\,h\,d(t,x),
\]
where $\psi^h \in H$ is the unique weak solution of the linearized adjoint system \eqref{linearized_adjoint} with $y$ as above, $F = L^*Ly$, and $\psi^T = D^*D(y(T,\cdot))$ (the strong $L^2(Q)$-convergence of the adjoint difference quotients to $\psi^h$, used implicitly here, is established in Appendix~\ref{D2} for arbitrary $f \in \mathcal{U}_{ad}$).

By the duality identity (derived exactly as the identity \eqref{CharControlStep2} of Step~1, with $f$ in place of $f^*$ and full justification in Appendix~\ref{D2}),
\begin{align*}
\int_Q \psi^h\, h\,d(t,x) &= \int_Q\big[L^*L y + a'(u)y\,\varphi_{xx} + 2(b'(u)y u_x\varphi + b(u)y_x\varphi)_x \\
&\qquad - b''(u)y u_x^2\varphi - 2b'(u)u_x y_x\varphi\big]\,y\,d(t,x) \;-\; \int_I |Dy(T,x)|^2\,dx \\
&=: I_1 + I_2 + I_3 + I_4 + I_5 + I_6 - \int_I |Dy(T)|^2\,dx.
\end{align*}

We estimate each term using the one-dimensional embeddings $H^1_0(I)\hookrightarrow L^\infty(I)$ and $\|v_x\|_{L^\infty(I)} \leq C\|v_{xx}\|$ for $v \in H^2(I)\cap H^1_0(I)$, the uniform bounds \eqref{uniform_u_bound}--\eqref{uniform_phi_bound}, and -- crucially -- the \emph{sharper} bound \eqref{Fprime_L2_bound} of Theorem~\ref{thm:gateaux_differentiability}, which only requires $h \in L^2(Q)$:
\begin{equation}\label{y_L2_bounds}
\|y\|_{L^\infty(Q)} + \|y\|_{L^\infty(0,T;H^1_0(I))} + \|y\|_{L^2(0,T;H^2(I))} + \|y_t\|_{L^2(Q)} \leq C\,\|h\|_{L^2(Q)}.
\end{equation}

For $I_1$: by Cauchy--Schwarz,
\[
|I_1| = \Big|\int_Q L^*L\,y\cdot y\,d(t,x)\Big| \leq \|L^*L\|_{\mathcal{L}(L^2(Q))}\,\|y\|_{L^2(Q)}^2 \leq C_1\,\|h\|_{L^2(Q)}^2.
\]

For $I_2$,
\[
|I_2| \leq \|y\|_{L^\infty(Q)}\,\|\varphi_{xx}\|_{L^2(Q)}\,\|y\|_{L^2(Q)} \leq C\,C_\varphi\,\|h\|_{L^2(Q)}^2.
\]

For $I_3 + I_4$, integration by parts in $x$ together with $y|_\Sigma = 0$ yields
\[
I_3 + I_4 = -\int_Q \big(b'(u)y u_x \varphi + b(u) y_x\varphi\big)\,y_x\,d(t,x).
\]
Bounding the integrand pointwise in $t$ and using $\|y\|_{L^\infty(I)},\|\varphi\|_{L^\infty(I)}\leq C\|y_x\|,\|\varphi_x\|$, $\|u_x\|_{L^\infty(I)}\leq C\|u_{xx}\|$, we get
\begin{align*}
|I_3 + I_4| &\leq C\int_0^T \big(\|y\|_{L^\infty(I)}\|u_x\|_{L^\infty(I)}\|\varphi\|_{L^\infty(I)} + \|y_x\|\|\varphi\|_{L^\infty(I)}\big)\|y_x\|\,dt \\
&\leq C\,\|y\|_{L^\infty(0,T;H^1_0(I))}\,\|\varphi\|_{L^\infty(0,T;H^1_0(I))}\,\big(\|u\|_{L^2(0,T;H^2)} + \sqrt{T}\big)\,\|y_x\|_{L^2(Q)}\\
&\leq C\,(C_u+1)\,C_\varphi\,\|h\|_{L^2(Q)}^2,
\end{align*}
where in the last step we used \eqref{y_L2_bounds} and the embedding $W \hookrightarrow C([0,T];H^1_0(I))$ to bound $\|\varphi\|_{L^\infty(0,T;H^1_0)}\leq C\|\varphi\|_W \leq C\,C_\varphi$.

For $I_5$,
\begin{align*}
|I_5| &\leq C\int_0^T \|y\|_{L^\infty(I)}^2\,\|u_x\|_{L^\infty(I)}^2\,\|\varphi\|_{L^\infty(I)}\,dt \\
&\leq C\,\|y\|_{L^\infty(Q)}^2\,\|\varphi\|_{L^\infty(0,T;H^1_0(I))}\,\|u\|_{L^2(0,T;H^2(I))}^2 \\
&\leq C\,C_\varphi\,C_u^2\,\|h\|_{L^2(Q)}^2.
\end{align*}

For $I_6$, similarly,
\begin{align*}
|I_6| &\leq C\int_0^T \|u_x\|_{L^\infty(I)}\,\|y_x\|\,\|\varphi\|_{L^\infty(I)}\,\|y\|_{L^\infty(I)}\,dt\\
&\leq C\,\|y\|_{L^\infty(Q)}\,\|\varphi\|_{L^\infty(0,T;H^1_0(I))}\,\|u\|_{L^2(0,T;H^2(I))}\,\|y_x\|_{L^2(Q)}\\
&\leq C\,C_\varphi\,C_u\,\|h\|_{L^2(Q)}^2.
\end{align*}

Combining all estimates and noting $\int_I |Dy(T,x)|^2\,dx \geq 0$,
\begin{equation}\label{psi_estimate}
\Big|\int_Q \psi^h\,h\,d(t,x)\Big| \leq k_0\,\|h\|_{L^2(Q)}^2,
\end{equation}
where
\[
k_0 = k_0\big(T,\,a_0,\,M,\,C_u,\,C_\varphi,\,\|L\|,\,\|D^*D\|_{\mathcal{L}(H^1_0)},\,p,\,q\big)
\]
is a constant that depends only on the problem data and on the uniform bounds $C_u, C_\varphi$ from \eqref{uniform_u_bound}--\eqref{uniform_phi_bound}, but \emph{not on the particular point} $f \in \mathcal{U}_{ad}$ where the second derivative is computed. We conclude that
\begin{equation} \label{uniform_J''_bound}
\langle J''(f),(h,h)\rangle \geq (\lambda - k_0)\,\|h\|_{L^2(Q)}^2 \qquad \forall\, f \in \mathcal{U}_{ad},\ \forall\, h \text{ with } f+h \in \mathcal{U}_{ad}.
\end{equation}

\smallskip
\noindent\emph{Step 2c: Uniqueness via Taylor's theorem.}
Suppose $\lambda > k_0$ and let $f, g \in \mathcal{U}_{ad}$. By the convexity of $\mathcal{U}_{ad}$, the segment $[f,g]_s := f + s(g-f)$ lies in $\mathcal{U}_{ad}$ for every $s \in [0,1]$. The map $s \mapsto J([f,g]_s)$ is twice continuously differentiable: indeed, the differentiability of $f \mapsto u(f)$ at every $f \in \mathcal{U}_{ad}$ is given by Theorem~\ref{thm:gateaux_differentiability}, and the differentiability of $f \mapsto \varphi(f)$ at every $f \in \mathcal{U}_{ad}$ follows from the same argument as in Appendix~\ref{D2}, applied with $f$ in place of $f^*$ as the base point (the proof uses only that $f^* \in \mathcal{U}_{ad}$, never the optimality of $f^*$, and the ingredients --- Lipschitz continuity of $f \mapsto u(f)$, well-posedness of the adjoint and linearized adjoint systems, Mean-Value Theorem --- are all available uniformly on $\mathcal{U}_{ad}$). Taylor's theorem with integral remainder then gives
\[
J(g) = J(f) + \langle J'(f), g-f\rangle + \int_0^1 (1-s)\,\langle J''([f,g]_s),(g-f, g-f)\rangle\,ds.
\]
Applying \eqref{uniform_J''_bound} pointwise in $s$ (with $f$ replaced by $[f,g]_s$ and $h$ replaced by $g-f$), we obtain
\begin{equation} \label{strict_convexity_estimate}
J(g) \geq J(f) + \langle J'(f), g-f\rangle + \frac{\lambda - k_0}{2}\,\|g-f\|_{L^2(Q)}^2.
\end{equation}
If $f^*$ minimizes $J$ over $\mathcal{U}_{ad}$, then $\langle J'(f^*), g-f^*\rangle \geq 0$ for every $g \in \mathcal{U}_{ad}$, and \eqref{strict_convexity_estimate} with $f = f^*$ yields
\[
J(g) - J(f^*) \geq \frac{\lambda - k_0}{2}\,\|g-f^*\|_{L^2(Q)}^2 > 0 \qquad \text{whenever } g \neq f^*.
\]
Hence $f^*$ is the unique minimizer of $J$ on $\mathcal{U}_{ad}$.

The fact that $(u^*, \varphi^*)$ is then the unique solution of the optimality system \eqref{optimality_system} follows from the equivalence between the projection formula and the variational inequality, together with the well-posedness of the state and adjoint systems (Theorem~\ref{thm:well_posedness} and Lemma~\ref{lem:adjoint_system}).
\end{proof}

\begin{remark}[Dependence of $k_0$ on the data] \label{rem:dependence_k}
The threshold $k_0$ depends on the data only through the a priori bounds $C_u$ and $C_\varphi$ from \eqref{uniform_u_bound}--\eqref{uniform_phi_bound}. Explicitly,
\begin{align*}
C_u &\leq C\big(\|u_0\|_{H^2(I)} + \delta\big), \\
C_\varphi &\leq C\Big(\|L^*L\|_{\mathcal{L}(L^2(Q))}\big(C_u + \|u_d\|_{L^2(Q)}\big) + \|D^*D\|_{\mathcal{L}(H^1_0(I))}\big(C_u + \|u_d^T\|_{H^1_0(I)}\big)\Big),
\end{align*}
where the bound on $C_\varphi$ uses Lemma~\ref{lem:adjoint_system} together with $\varphi^T \in H^1_0(I)$ supplied by assumption (\textbf{A6}). Thus $k_0$ depends only on the problem data $u_0$, $u_d$, $u_d^T$, $L$, $D$, $T$, $a_0$, $M$, $p$, $q$, $\delta$, $|I|$, and not on the unknown optimal pair $(u^*, \varphi^*, f^*)$.
\end{remark}

\begin{remark}[Interpretation of optimality system] \label{rem:optimality_interpretation}
The optimality system \eqref{optimality_system} exhibits the characteristic three-part structure typical of optimal control problems for parabolic PDEs: the state equation evolves forward in time and describes the underlying physical dynamics; the adjoint equation evolves backward in time and encodes sensitivity information about how changes in the state affect the cost functional; and the projection formula establishes an explicit relationship between the optimal control and the adjoint variable.
\end{remark}

\begin{remark}[Role of the regularization parameter] \label{rem:lambda_role}
The condition $\lambda > k_0$ ensures the strict convexity of $J$ on $\mathcal{U}_{ad}$. In practice, $\lambda$ is interpreted as a Tikhonov regularization parameter that balances the tracking objective against the control cost. Large values of $\lambda$ lead to smaller control magnitudes but potentially poorer tracking performance.

Unlike linear-quadratic problems, where uniqueness holds for every $\lambda > 0$, the threshold $\lambda > k_0$ in Theorem~\ref{thm:optimality_conditions} reflects the non-convexity of $f \mapsto J(f)$ induced by the quasilinear state equation \eqref{MainModel}: the second variation contains lower-order curvature contributions through $\psi$ that can be negative, and only when $\lambda$ is large enough does the convex Tikhonov term dominate. As made explicit in Remark~\ref{rem:dependence_k}, the constant $k_0$ depends only on the data (in particular on the smallness parameter $\delta$ via the energy bound $C_u$ on the state, and on the size of the targets $u_d, u_d^T$ via the adjoint bound $C_\varphi$), and \emph{not} on the unknown optimal pair. The bound $k_0$ produced by our analysis is conservative; the experience of related works \cite{bonifacius2018second,bonifacius2025optimal} suggests that the practical threshold can be significantly smaller than the worst-case constant obtained by tracking the chain of energy estimates.
\end{remark}


The characterization theorem has several important consequences.

\begin{corollary}[Gradient formula] \label{cor:gradient}
Under the hypotheses of Theorem~\ref{thm:optimality_conditions}, the gradient of the reduced cost functional $\widehat{J}(f) := J(f)$ (with the constraint $u = u(f)$) is given by
\[
\nabla \widehat{J}(f) = \varphi(f) + \lambda f,
\]
where $\varphi(f)$ solves the adjoint system with state $u(f)$.
\end{corollary}

\section{Conclusions and future perspectives} \label{sec:conclusions}

This work has established a comprehensive theoretical framework for the optimal control of quasilinear parabolic equations modeling nonlinear heat conduction phenomena in one spatial dimension. Our main contributions include the development of a well-posedness theory for the controlled system under general nonlinearity assumptions, requiring careful energy estimates to handle the challenging quadratic gradient term. We established the G\^ateaux differentiability of the control-to-state mapping through the analysis of a non-standard linearized system, and proved existence and uniqueness results for optimal controls characterized by a coupled forward-backward optimality system.

The one-dimensional setting was essential to our analysis, permitting the use of strong Sobolev embeddings, compactness results via the Aubin--Lions lemma, and enhanced regularity properties that enabled rigorous treatment of the nonlinear terms. While these results provide a solid foundation, certain limitations remain, including the local nature of well-posedness requiring smallness conditions and the dependence on sufficient regularization for uniqueness (via the threshold $\lambda > k_0$ of Theorem~\ref{thm:optimality_conditions}). The extension of the present framework to $d=2$ and $d=3$ stands out as the main open problem we leave for future work.

The two- and three-dimensional cases are genuinely open and, in our view, the most interesting direction in which this work can be continued. Several aspects of the present analysis are dimension-independent and should carry over: the structure of the optimality system, the formal computation of the adjoint, the projection formula $f^* = P_{\mathcal{U}_{ad}}(-\varphi^*/\lambda)$, and the second-variation argument leading to the threshold $\lambda > k_0$. What does \emph{not} carry over is the chain of energy estimates: the embedding $H^1(\Omega)\hookrightarrow L^\infty(\Omega)$ fails in $d\geq 2$, and the quadratic gradient term $b(u)|\nabla u|^2$ then sits exactly at the critical regularity threshold of the energy method. A natural starting point is the gradient-constrained framework of \cite{bonifacius2025optimal}, which controls the quadratic gradient term by imposing pointwise (or integrated) constraints on $\nabla u$ in the admissible set. A side question --- of independent interest --- is the behaviour of the smallness threshold $\delta_0$ and of the regularization threshold $k_0$ as one passes from $d=1$ to $d\geq 2$: one would like to understand whether these constants degenerate in a controlled way and, if so, what the resulting effective dimension-dependent theory looks like.

Several other promising research directions emerge from this work. The gradient formula of Corollary~\ref{cor:gradient} makes the optimality system \eqref{optimality_system} amenable to a numerical treatment, which we intend to pursue elsewhere. Related problems worthy of investigation include boundary control formulations, parameter identification frameworks (where the unknown is one or both of the nonlinear coefficients $a, b$), stochastic optimal control incorporating uncertainty in the source $f$ or the initial data $u_0$, and multi-physics coupling with other physical phenomena that share the structural feature of a quadratic gradient nonlinearity.

The interplay between PDE theory, functional analysis, and optimization methods demonstrated in this work highlights the rich mathematical structure of optimal control problems for nonlinear partial differential equations. While the one-dimensional setting provided a tractable framework for developing the core theory, the techniques established here serve as a foundation for addressing more challenging higher-dimensional problems and can be adapted to other nonlinear PDE models arising in physics and engineering.


\section*{Acknowledgments}
The first author gratefully acknowledges the hospitality of the Chair for Dynamics,
Control, Machine Learning and Numerics (Alexander von Humboldt Professorship) at
Friedrich--Alex\-an\-der--Uni\-ver\-si\-t\"at Erlangen--N\"urn\-berg, where part of this
work was carried out, and the support of a CAPES (Brazil) scholarship, grant
\mbox{88887.895023/2023-00}, under the Brazil--Germany PROBRAL programme (CAPES/DAAD).

\appendix

\section{Detailed proof of Theorem~\ref{thm:well_posedness}} \label{A}

Let $\{\psi_j\}_{j=1}^\infty$ be the eigenfunctions of the Dirichlet Laplacian on $I$ satisfying
\[
-\psi_j'' = \lambda_j \psi_j \quad \text{in } I, \qquad \psi_j = 0 \quad \text{on } \partial I,
\]
with $0 < \lambda_1 \leq \lambda_2 \leq \cdots \to \infty$. Let $V_n := \mathrm{span}\{\psi_1, \dots, \psi_n\}$, and let $u_{0n} := \sum_{j=1}^n (u_0, \psi_j) \psi_j$ be the $L^2(I)$-projection of $u_0$ onto $V_n$.

We seek approximate solutions of the form
\[
u_n(t,x) = \sum_{j=1}^n g_{j,n}(t) \psi_j(x)
\]
satisfying, for all $w \in V_n$,
\begin{equation} \label{FiniteDimModel}
\begin{cases}
(u_{n,t}, w) + (a(u_n)u_{n,x}, w_x) + (b(u_n)u_{n,x}^2, w) = (f, w), & t\in (0,T_n),\\
(u_n(0,\cdot), w) = (u_0, w).
\end{cases}
\end{equation}

By Carath\'eodory's theorem, there exists $T_n > 0$ and absolutely continuous functions $g_{j,n}$ satisfying \eqref{FiniteDimModel}.

\begin{center}
\textbf{ESTIMATE I (basic energy)}
\end{center}

Taking $w = u_n$ in \eqref{FiniteDimModel} and using
\(
(u_{n,t}, u_n) = \tfrac{1}{2}\tfrac{d}{dt}\|u_n\|^2,
\)
\(
(a(u_n)u_{n,x}, u_{n,x}) \geq a_0\|u_{n,x}\|^2,
\)
\(
(b(u_n)u_{n,x}^2, u_n) = \int_I b(u_n)u_n\, u_{n,x}^2\, dx \geq 0
\)
(by the sign condition $b(s)s\geq 0$), and Young's inequality, we obtain
\[
\frac{d}{dt}\|u_n\|^2 + 2 a_0 \|u_{n,x}\|^2 \leq \|f\|^2 + \|u_n\|^2.
\]
Gronwall's lemma and integration in time yield
\begin{equation} \label{estimate1}
\sup_{t\in[0,T]}\|u_n(t)\|^2 + 2a_0\int_0^T\|u_{n,x}\|^2\, dt \leq e^{T}\big(\|u_0\|^2 + \|f\|_{L^2(Q)}^2\big),
\end{equation}
which holds independently of $n$. In particular, the local existence on $[0,T_n]$ extends to $[0,T]$.

\begin{center}
\textbf{ESTIMATE II (time derivative)}
\end{center}

Taking $w = u_{n,t}$ in \eqref{FiniteDimModel} gives
\begin{equation} \label{utAsATestFn}
\|u_{n,t}\|^2 + (a(u_n)u_{n,x}, u_{n,tx}) + (b(u_n)u_{n,x}^2, u_{n,t}) = (f, u_{n,t}).
\end{equation}

For the first nonlinear term:
\begin{align*}
(a(u_n)u_{n,x}, u_{n,tx}) &= \frac{d}{dt}\!\left(\frac{1}{2}\|a(u_n)^{1/2}u_{n,x}\|^2\right) - \frac{1}{2}(a'(u_n)u_{n,t}, u_{n,x}^2) \\
&\geq \frac{d}{dt}\!\left(\frac{1}{2}\|a(u_n)^{1/2}u_{n,x}\|^2\right) - \frac{M}{2}\!\int_I (1+|u_n|^{p-1})|u_{n,t}|u_{n,x}^2\, dx.
\end{align*}
The remainder is estimated as follows. Combining the one-dimensional Sobolev embedding $H^1_0(I) \hookrightarrow L^\infty(I)$ with the resulting $\|u_n\|_{L^\infty}\leq C\|u_{n,x}\|$ and $\|u_{n,t}\|_{L^\infty}\leq C\|u_{n,xt}\|$, we obtain
\begin{equation} \label{HolderSobolevBound1}
\begin{aligned}
\int_I (1+|u_n|^{p-1})|u_{n,t}| u_{n,x}^2\, dx
&\leq (1 + \|u_n\|_{L^\infty}^{p-1})\,\|u_{n,t}\|_{L^\infty}\,\|u_{n,x}\|^2 \\
&\leq C\,(1+\|u_{n,x}\|^{p-1})\,\|u_{n,xt}\|\,\|u_{n,x}\|^2.
\end{aligned}
\end{equation}
This pattern --- pulling out one $L^\infty(I)$-factor via the Sobolev embedding and keeping the remaining $L^2$-norms --- will be used repeatedly throughout this appendix without further comment.

For the second nonlinear term, by (\textbf{A2}) and the same argument as in \eqref{HolderSobolevBound1} (now bounding $|u_n|^q \leq \|u_n\|^q_{L^\infty} \leq C\|u_{n,x}\|^q$ and $|u_{n,t}| \leq \|u_{n,t}\|_{L^\infty} \leq C\|u_{n,xt}\|$, the latter pulled out of the integral):
\begin{equation}\label{HolderSobolevBound2}
\big|(b(u_n)u_{n,x}^2, u_{n,t})\big| \leq M\!\int_I (1+|u_n|^q)\, u_{n,x}^2\, |u_{n,t}|\, dx \leq C\,\|u_{n,xt}\|\,\|u_{n,x}\|^2\big(1+\|u_{n,x}\|^q\big).
\end{equation}

Substituting back, applying Young's inequality with parameters $\eta, \epsilon > 0$, and absorbing terms,
\begin{equation} \label{eq:Est2_step1}
\begin{aligned}
\frac{d}{dt}\!\left(\frac{\|a(u_n)^{1/2}u_{n,x}\|^2}{2}\right) + \|u_{n,t}\|^2
&\leq C\, \|u_{n,xt}\|^2\,P_1(\|u_{n,x}\|) \\
&\quad + C\big(\|f\|^2 + \|u_{n,x}\|^2 + \|u_{n,t}\|^2\big),
\end{aligned}
\end{equation}
where $P_1$ is a polynomial with $P_1(0)=0$ (whose explicit form is $P_1(x) = \frac{M^2}{2\epsilon}x^{2}(1+x^q)^2 + \frac{M^2}{2\epsilon}(1+x^{p-1})^2 x^2$).

\begin{center}
\textbf{ESTIMATE II (cont., time-differentiated)}
\end{center}

Differentiating \eqref{FiniteDimModel} with respect to $t$ and taking $w = u_{n,t}$:
\begin{equation}\label{derivative}
\begin{split}
\frac{d}{dt}\!\left(\frac{\|u_{n,t}\|^{2}}{2}\right)
&+ \int_{I} \big[a(u_n) u_{n,xt}^{2} + a'(u_n) u_{n,t} u_{n,x} u_{n,xt}\big]\, dx \\
&+ \int_{I} \big[b'(u_n) u_{n,t}^{2} u_{n,x}^{2} + 2b(u_n) u_{n,x} u_{n,xt} u_{n,t}\big]\, dx = (f_t, u_{n,t}).
\end{split}
\end{equation}

Term-by-term estimation, using $a(u_n)\geq a_0$ and the embeddings, gives
\begin{align*}
&\int_I a(u_n) u_{n,xt}^2\, dx + \int_I a'(u_n) u_{n,t} u_{n,x} u_{n,xt}\, dx \\
&\quad \geq a_0\|u_{n,xt}\|^2 - C(1+\|u_{n,x}\|^{p-1})\|u_{n,x}\|\,\|u_{n,xt}\|^2,\\
&\big|\int_I b'(u_n)u_{n,t}^2 u_{n,x}^2\, dx + 2\int_I b(u_n)u_{n,x}u_{n,xt}u_{n,t}\, dx\big| \\
&\quad \leq C\|u_{n,xt}\|^2\big(\|u_{n,x}\|^{q+1} + \|u_{n,x}\|^2\big),\\
&|(f_t, u_{n,t})| \leq \frac{1}{2\eta}\|f_t\|^2 + \frac{\eta}{2}\|u_{n,t}\|^2.
\end{align*}
Combining,
\begin{equation}\label{74}
\frac{d}{dt}\!\left(\frac{\|u_{n,t}\|^{2}}{2}\right) + a_0\|u_{n,xt}\|^{2} \leq C\,\|u_{n,xt}\|^{2}\, P_2(\|u_{n,x}\|) + \frac{1}{2\eta}\|f_t\|^{2} + \frac{\eta}{2}\|u_{n,t}\|^{2},
\end{equation}
where $P_2(x) := 3M x^{q+1} + Mx^2 + M(1+x^{p-1})x$ satisfies $P_2(0)=0$.

Adding \eqref{eq:Est2_step1} and \eqref{74} and choosing $\eta$ small,
\begin{equation} \label{AlmostEstimateII}
\begin{aligned}
\frac{1}{2}\frac{d}{dt}\!\big(\|u_{n,t}\|^{2} + \|a(u_n)^{1/2}u_{n,x}\|^{2}\big) + \|u_{n,t}\|^{2} + a_0\|u_{n,xt}\|^{2}
&\leq \|u_{n,xt}\|^{2}P(\|u_{n,x}\|) \\
&\quad + C(\|f\|^{2} + \|f_t\|^{2}),
\end{aligned}
\end{equation}
where $P := P_1 + P_2$ is continuous with $P(0) = 0$.

\medskip\noindent\textbf{Bootstrapping (continuation argument).}
Since $P$ is continuous and $P(0) = 0$, there exists $\kappa > 0$ such that $P(x) \leq a_0/2$ for all $x \in [0, \kappa]$. Define
\[
t^* := \sup\Big\{\, \tau \in [0,T] :\ \|u_{n,x}(s,\cdot)\| \leq \kappa \ \forall\, s \in [0,\tau]\, \Big\},
\]
with the convention $\sup\emptyset = 0$.

\textit{Step (a): $t^* > 0$.} Since $u_n \in C^1([0,T];V_n)$, the map $s \mapsto \|u_{n,x}(s,\cdot)\|$ is continuous. Assume the smallness of initial data
\begin{equation} \label{SmallnessOfInitialData}
\|u_{0,x}\| \leq \kappa/2.
\end{equation}
Then $\|u_{n,x}(0,\cdot)\| = \|u_{0n,x}\| \leq \|u_{0,x}\| \leq \kappa/2 < \kappa$, so by continuity there is a neighborhood of $0$ where $\|u_{n,x}\| \leq \kappa$, hence $t^* > 0$.

\textit{Step (b): $t^* = T$.} Suppose for contradiction $t^* < T$. By definition of $t^*$ and continuity, $\|u_{n,x}(t^*,\cdot)\| = \kappa$. On $[0, t^*]$, $P(\|u_{n,x}\|) \leq a_0/2$, so \eqref{AlmostEstimateII} reduces to
\[
\frac{1}{2}\frac{d}{dt}\!\big(\|u_{n,t}\|^{2} + \|a(u_n)^{1/2}u_{n,x}\|^{2}\big) + \|u_{n,t}\|^{2} + \frac{a_0}{2}\|u_{n,xt}\|^{2} \leq C(\|f\|^{2} + \|f_t\|^{2}).
\]
Integrating from $0$ to $t^*$ and using $a(u_n) \geq a_0$, $\|a(u_n)^{1/2}u_{n,x}\|^2 \geq a_0 \|u_{n,x}\|^2$,
\begin{align} \label{AlmostEstIIbis}
\frac{a_0}{2}\|u_{n,x}(t^*,\cdot)\|^{2} &\leq Q_0(\|u_0\|, \|u_{0,x}\|, \|u_{0,xx}\|) + C\big(\|f\|_{L^2(Q)}^{2} + \|f_t\|_{L^2(Q)}^{2}\big),
\end{align}
where $Q_0$ is a polynomial in the initial data norms. By the smallness of the data (\textbf{A3})--(\textbf{A4})--(\textbf{A5}), specifically $\|f\|_{L^2(Q)}^2 + \|f_t\|_{L^2(Q)}^2 \leq \delta'^2$ where
\[
\delta'^2 := \frac{(a_0\kappa)^2}{8(8a_0 + e^T)},
\]
and choosing $\|u_{0,xx}\|$ small enough to ensure $Q_0(\|u_0\|, \|u_{0,x}\|, \|u_{0,xx}\|) \leq a_0\kappa^2/8$, we obtain
\begin{equation} \label{ContradictionI}
\|u_{n,x}(t^*,\cdot)\|^2 \leq \frac{\kappa^2}{2},\qquad\text{hence}\qquad \|u_{n,x}(t^*,\cdot)\| \leq \frac{\kappa}{\sqrt{2}} < \kappa.
\end{equation}
But by definition of $t^*$, $\|u_{n,x}(t^*,\cdot)\| = \kappa$. This contradiction shows $t^* = T$.

\textit{Step (c): Uniform estimate.} It now follows from \eqref{AlmostEstimateII} on $[0,T]$ that
\begin{equation} \label{EstimateII}
\begin{aligned}
\sup_{t\in[0,T]}\big(\|u_{n,t}\|^{2} + \|u_{n,x}\|^2\big)
&+ \int_{0}^{T}\big(\|u_{n,t}\|^{2} + \|u_{n,xt}\|^{2}\big)\,dt \\
&\leq C\big(\|u_0\|_{H^1}^2 + \|f\|_{L^2(Q)}^{2} + \|f_t\|_{L^2(Q)}^{2}\big).
\end{aligned}
\end{equation}

\begin{center}
\textbf{ESTIMATE III (spatial regularity)}
\end{center}

Plugging $w = -u_{n,xx}$ in \eqref{FiniteDimModel} (or, more precisely, taking $w$ such that $w_x = -u_{n,xxx}$ and exploiting the basis structure to integrate by parts):
\[
K_1 + K_2 - K_3 = -K_4,
\]
where
\begin{align*}
K_1 &:= -\int_I u_{n,t} u_{n,xx}\,dx = \frac{1}{2}\frac{d}{dt}\|u_{n,x}\|^2,\\
K_2 &:= \int_I (a(u_n) u_{n,xx} + a'(u_n) u_{n,x}^2) u_{n,xx}\, dx \\
&\qquad \geq a_0 \|u_{n,xx}\|^2 - C(1+\|u_n\|_{L^\infty}^{p-1})\|u_{n,x}\|\, \|u_{n,xx}\|^2,\\
K_3 &:= \int_I b(u_n) u_{n,x}^2 u_{n,xx}\, dx,\\
|K_3| &\leq M(1+\|u_n\|_{L^\infty}^q)\|u_{n,x}\|\,\|u_{n,xx}\|^2,\\
K_4 &:= \int_I f\, u_{n,xx}\,dx,\\
|K_4| &\leq \frac{1}{a_0}\|f\|^2 + \frac{a_0}{4}\|u_{n,xx}\|^2.
\end{align*}
Combining,
\begin{equation}
\frac{1}{2}\frac{d}{dt}\|u_{n,x}\|^2 + \big(\tfrac{a_0}{2} - \widetilde P(\|u_{n,x}\|)\big)\|u_{n,xx}\|^2 \leq \frac{1}{a_0}\|f\|^2,
\end{equation}
where $\widetilde P(x) = M x(2 + x^{p-1} + x^q)$. Possibly shrinking $\kappa, \delta$ further, we ensure $\widetilde P(\|u_{n,x}\|) \leq a_0/4$, hence
\begin{equation} \label{WPEstimateFour}
\sup_{t\in[0,T]}\|u_{n,x}\|^2 + \|u_{n,xx}\|_{L^2(Q)}^2 \leq C\big(\|u_{0,x}\|^2 + \|f\|_{L^2(Q)}^2\big).
\end{equation}

\medskip
Combining \eqref{estimate1}, \eqref{EstimateII}, and \eqref{WPEstimateFour},
\[
\|u_n\|_W \leq C\big(\|u_0\|_{H^2(I)} + \|f\|_{H^1(0,T;L^2(I))}\big).
\]
By the Aubin--Lions--Simon Theorem \cite{lions1969quelques,aubin1963analyse,simon1986compact}, the embeddings
\[
W \hookrightarrow L^2(0,T;H^1_0(I))\quad \text{and}\quad W \hookrightarrow C([0,T]; L^2(I))
\]
are compact. We extract a subsequence of $\{u_n\}$ converging weakly in $W$ and strongly in $L^2(0,T;H^1_0(I))\cap C([0,T]; L^2(I))$. Passage to the limit in the nonlinear terms is performed as in the proof of Theorem~\ref{thm:existence_optimal} (continuity of $a, b$, dominated convergence, and Aubin--Lions on the gradient family). For further standard details we refer to \cite{lions1969quelques,lions2012non,evans2010partial}.

\section{Detailed proof of Lemma~\ref{lem:linearized_system}} \label{B}

We use the Galerkin method with the Hilbert basis $\{\psi_j\}_{j=1}^\infty$ of $H^1_0(I)$ given by the eigenfunctions of the Dirichlet Laplacian (as in Appendix~\ref{A}). Let $V_n := \mathrm{span}\{\psi_1,\dots,\psi_n\}$ and let $\pi_n$ denote the $L^2(I)$-projection onto $V_n$.

Fix $n \geq 1$. By Carath\'eodory's existence theorem \cite{coddington1955theory}, there exists $T_n \in (0,T]$ and measurable functions $c_{j,n} : [0,T_n] \to \mathbb{R}$ such that
\[
y_n := \sum_{j=1}^n c_{j,n}\,\psi_j
\]
solves
\begin{equation} \label{LinFiniteDimModel}
\begin{cases}
(y_{n,t}, \phi) + ((a(u)y_n)_x, \phi_x) + (2 b(u) u_x y_{n,x}, \phi) + (b'(u) u_x^2 y_n, \phi) = (h_n, \phi),\\
y_n(0,\cdot) = 0,
\end{cases}
\end{equation}
for all $\phi \in V_n$ and $t \in [0,T_n]$, where $h_n := \pi_n h$.

Taking $\phi = y_n$ (we drop the subscript $n$ for clarity in the energy estimates):
\begin{equation} \label{LinSystem}
\frac{1}{2}\frac{d}{dt}\|y\|^2 + \int_I (a(u)y)_x y_x\, dx + 2\int_I b(u) u_x y_x y\, dx + \int_I b'(u) u_x^2 y^2\, dx = \int_I h\,y\, dx.
\end{equation}

Estimating the diffusion term:
\begin{equation} \label{Step1LinSystem}
\begin{aligned}
\int_I (a(u)y)_x y_x\, dx &= \int_I a(u) y_x^2\, dx + \int_I a'(u) u_x\, y\, y_x\, dx\\
&\geq a_0\|y_x\|^2 - C\|u_x\|_{L^\infty}\|y\|\|y_x\| \\
&\geq a_0\|y_x\|^2 - C\|u_{xx}\|\|y\|\|y_x\|\\
&\geq \frac{a_0}{2}\|y_x\|^2 - C\|u_{xx}\|^2\|y\|^2.
\end{aligned}
\end{equation}

The other nonlinear terms:
\begin{equation} \label{Step2LinSystem}
\Big|\int_I b(u) u_x y_x y\, dx\Big| \leq C\|u_{xx}\|\|y_x\|\|y\| \leq \epsilon\|y_x\|^2 + C_\epsilon \|u_{xx}\|^2\|y\|^2,
\end{equation}
\begin{equation} \label{Step3LinSystem}
\Big|\int_I b'(u) u_x^2 y^2\, dx\Big| \leq C\|u_x\|_{L^\infty}^2\|y\|^2 \leq C\|u_{xx}\|^2\|y\|^2.
\end{equation}
The right-hand side: $|(h, y)| \leq \frac{1}{2}\|h\|^2 + \frac{1}{2}\|y\|^2$.

Choosing $\epsilon = a_0/8$ and combining,
\begin{equation} \label{TreatedLinSystem}
\frac{d}{dt}\|y\|^2 + a_0\|y_x\|^2 \leq C\big(1 + \|u_{xx}\|^2\big)\|y\|^2 + \|h\|^2.
\end{equation}
Note: all constants $C$ here are independent of $n$. Gronwall's lemma yields
\begin{equation} \label{EstLinSystemPt1}
\sup_{t\in[0,T]}\|y\|^2 \leq C\,e^{C\int_0^T (1+\|u_{xx}\|^2)\, dt}\|h\|_{L^2(Q)}^2 \leq C(\|u\|_W)\|h\|_{L^2(Q)}^2.
\end{equation}
This estimate also implies $T_n = T$. Integrating \eqref{TreatedLinSystem} in time and using \eqref{EstLinSystemPt1},
\begin{equation} \label{EstLinSystemPt2}
\int_0^T \|y_x\|^2\, dt \leq C(\|u\|_W)\,\|h\|_{L^2(Q)}^2.
\end{equation}

For the time derivative: for $\phi \in H^1_0(I)$,
\begin{align*}
|(y_t, \phi)| &= |(y_t, \pi_n\phi)| \\
&= \big| -(a(u)y_x, (\pi_n\phi)_x) - (a'(u)u_x y, (\pi_n\phi)_x) \\
&\quad + (-2b(u)u_x y_x - b'(u)u_x^2 y + h_n, \pi_n\phi)\big|\\
&\leq C\|h\|_{L^2(Q)}\|\phi\|_{H^1_0},
\end{align*}
whence
\begin{equation} \label{EstLinSystemPt3}
\|y_t\|_{L^2(0,T;H^{-1}(I))} \leq C\|h\|_{L^2(Q)}.
\end{equation}

The estimates \eqref{EstLinSystemPt1}--\eqref{EstLinSystemPt3} allow us to extract a subsequence $\{y_{n_k}\}$ converging to a weak solution $y \in H$. We now establish the improved $W$-regularity, distinguishing the two bounds stated in Lemma~\ref{lem:linearized_system}.

\smallskip
\noindent\emph{Sharper bound \eqref{linearized_estimate_L2} ($h\in L^2(Q)$ alone).}
Testing the projected equation \eqref{LinFiniteDimModel} successively with $-y_{n,xx}$ (spatial regularity) and with $y_{n,t}$ (time regularity) and proceeding exactly as in Estimates~II and~III of the proof of Theorem~\ref{thm:gateaux_differentiability} (with $u_\epsilon$ replaced by $u$, $y_\epsilon$ by $y_n$, and the source term $h$ kept on the right-hand side), one obtains
\begin{equation*}
\sup_{t\in[0,T]}\big(\|y\|^2 + \|y_x\|^2\big) + \|y_{xx}\|_{L^2(Q)}^2 + \|y_t\|_{L^2(Q)}^2 \leq C(\|u\|_W)\,\|h\|_{L^2(Q)}^2.
\end{equation*}
Note that this bound \emph{does not require} differentiating the equation in time, hence requires only $h\in L^2(Q)$. By the one-dimensional embedding $L^\infty(0,T;H^1_0(I))\hookrightarrow L^\infty(Q)$, this also yields $\|y\|_{L^\infty(Q)} \leq C\|h\|_{L^2(Q)}$, completing the proof of \eqref{linearized_estimate_L2}--\eqref{linearized_estimate_Linf}.

\smallskip
\noindent\emph{Full $W$-bound \eqref{linearized_estimate}, for $h\in H^1(0,T;L^2(I))$.}
To obtain in addition the bound $\|y_t\|_{L^2(0,T;H^1_0(I))}$, we differentiate the equation in time and test with $y_{n,t}$, mirroring Estimate~IV of the proof of Theorem~\ref{thm:gateaux_differentiability}. This produces a source term involving $h_t$, which is precisely where the additional regularity $h_t\in L^2(Q)$ enters. The resulting bound is
\[
\sup_{t\in[0,T]}\|y_t\|^2 + \|y_{tx}\|_{L^2(Q)}^2 \leq C(\|u\|_W)\,\|h\|_{H^1(0,T;L^2(I))}^2,
\]
and combining with the sharper bound above gives
\[
\|y\|_W \leq C(\|u\|_W)\,\|h\|_{H^1(0,T;L^2(I))}.
\]
The detailed term-by-term estimates parallel Estimates~II--IV of the proof of Theorem~\ref{thm:gateaux_differentiability} (noting that, since $u_\epsilon \equiv u$ in the present linear problem, all the singular difference quotients $(a(u_\epsilon)-a(u))/\epsilon$ etc.\ are absent and the calculation simplifies considerably); we omit the mechanical details. Uniqueness of the strong solution follows from the linearity of the system and the energy estimate applied to the difference of two solutions.

\section{Proof of Lemma~\ref{lem:adjoint_system}}\label{C}

We establish the well-posedness of the adjoint system
\begin{equation*}
\begin{cases}
-\varphi_t - a(u)\varphi_{xx} - 2(b(u)u_x \varphi)_x + b'(u)u_x^2 \varphi = G, & \text{in } Q,\\
\varphi = 0, & \text{on } \Sigma,\\
\varphi(T,x) = \varphi^T(x), & \text{for } x \in I,
\end{cases}
\end{equation*}
where $u \in W$ is given, $G \in L^2(Q)$, and $\varphi^T \in H^1_0(I)$.

By a time-reversal $\widetilde\varphi(t,x) := \varphi(T-t,x)$, this system is equivalent to a forward-in-time linear parabolic system. Standard parabolic theory \cite{lions2012non,evans2010partial} provides the existence of a unique weak solution $\varphi \in H$ (where $H$ is as in \eqref{defH}) satisfying
\[
\|\varphi\|_H \leq C\big(\|G\|_{L^2(Q)} + \|\varphi^T\|_{L^2(I)}\big).
\]
We focus on the improved regularity $\varphi \in W$, which we obtain through energy estimates carried out at the Galerkin level.

As in Appendices~\ref{B} and~\ref{D}, we employ the Galerkin method with the Hilbert basis $\{\psi_j\}_{j=1}^\infty$ of $H^1_0(I)$ given by the eigenfunctions of the Dirichlet Laplacian (as in Appendix~\ref{A}). Let $V_n := \mathrm{span}\{\psi_1,\dots,\psi_n\}$ and let $\pi_n$ denote the $L^2(I)$-projection onto $V_n$. By Carath\'eodory's existence theorem \cite{coddington1955theory}, there exists $T_n \in (0,T]$ and measurable functions $e_{j,n} : [0,T_n] \to \mathbb{R}$ such that
\[
\varphi_n(t,x) := \sum_{j=1}^n e_{j,n}(t)\,\psi_j(x)
\]
solves the projected (time-reversed) system on $V_n$ with terminal data $\varphi_n(T) = \pi_n\varphi^T$ and source $G_n := \pi_n G$. All the estimates below are performed on $\varphi_n$, where the integrands are square-integrable (in particular $\varphi_{n,xx} \in V_n \subset L^2(I)$, so the integrals $\int_I G_n\varphi_{n,xx}\,dx$ and $\int_I (b(u)u_x\varphi_n)_x\varphi_{n,xx}\,dx$ are well-defined); the regularity $\varphi \in W$ then follows from the resulting uniform-in-$n$ bounds and a standard limiting argument. For simplicity we drop the subscript $n$ in the energy estimates that follow.

Throughout the proof we use, in addition to the standard one-dimensional embedding $\|v\|_{L^\infty(I)} \leq C\|v_x\|$ for $v \in H^1_0(I)$, the \emph{sharper} Gagliardo--Nirenberg interpolation $\|v_x\|_{L^\infty(I)}^2 \leq C\,\|v_{xx}\|\,\|v_x\|$ for $v \in H^2(I)\cap H^1_0(I)$. Combined with the uniform-in-$t$ bound $\|u_x(t,\cdot)\| \leq C\|u\|_W \leq C$ provided by $W \hookrightarrow C([0,T]; H^1_0(I))$, this gives the \emph{linear} bound
\begin{equation}\label{AppC_GNsharp}
\|u_x(t,\cdot)\|_{L^\infty(I)}^2 \leq C\,\|u_{xx}(t,\cdot)\|,
\end{equation}
which is critical for keeping the Gronwall coefficients below in $L^1(0,T)$.

\smallskip
\noindent\emph{Spatial regularity.}
Test the projected equation with $-\varphi_{xx}$ (well-defined since $\varphi_{xx} \in V_n \subset L^2(I)$):
\begin{align*}
\frac{1}{2}\frac{d}{dt}\|\varphi_x\|^2 &+ \int_I a(u) \varphi_{xx}^2\, dx \\
&= -\int_I a'(u) u_x \varphi_x \varphi_{xx}\, dx + 2\int_I (b(u) u_x \varphi)_x\, \varphi_{xx}\, dx \\
&\quad - \int_I b'(u) u_x^2 \varphi \varphi_{xx}\, dx + \int_I G \varphi_{xx}\, dx.
\end{align*}
(Here the time-direction is reversed; we tacitly assume the time-reversal has been applied.)

Each term is estimated using the 1D embeddings:
\begin{align*}
\Big|\int_I a'(u) u_x \varphi_x \varphi_{xx}\, dx\Big| &\leq C\|u_x\|_{L^\infty}\|\varphi_x\|\|\varphi_{xx}\| \\
&\leq C\,\|u_{xx}\|^{1/2}\,\|\varphi_x\|\,\|\varphi_{xx}\| \\
&\leq \tfrac{a_0}{8}\|\varphi_{xx}\|^2 + C\|u_{xx}\|\|\varphi_x\|^2,
\end{align*}
where we used \eqref{AppC_GNsharp} in the second inequality. For the second term, expand the derivative:
\begin{align*}
2\int_I (b(u) u_x \varphi)_x \varphi_{xx}\, dx &= 2\int_I [b'(u) u_x^2 \varphi + b(u) u_{xx} \varphi + b(u) u_x \varphi_x]\varphi_{xx}\, dx.
\end{align*}
We bound each subterm using $\|\varphi\|_{L^\infty} \leq C\|\varphi_x\|$ and \eqref{AppC_GNsharp}:
\begin{align*}
\Big|\int_I b'(u) u_x^2 \varphi \varphi_{xx}\, dx\Big| &\leq C\|u_x\|_{L^\infty}^2\|\varphi\|_{L^\infty}\|\varphi_{xx}\| \\
&\leq C\,\|u_{xx}\|\,\|\varphi_x\|\,\|\varphi_{xx}\| \\
&\leq \tfrac{a_0}{16}\|\varphi_{xx}\|^2 + C\|u_{xx}\|^2\|\varphi_x\|^2,\\
\Big|\int_I b(u) u_{xx} \varphi \varphi_{xx}\, dx\Big| &\leq C\|u_{xx}\|\,\|\varphi\|_{L^\infty}\,\|\varphi_{xx}\|\\
&\leq C\,\|u_{xx}\|\,\|\varphi_x\|\,\|\varphi_{xx}\| \\
&\leq \tfrac{a_0}{16}\|\varphi_{xx}\|^2 + C\|u_{xx}\|^2\|\varphi_x\|^2,\\
\Big|\int_I b(u) u_x \varphi_x \varphi_{xx}\, dx\Big| &\leq C\|u_x\|_{L^\infty}\,\|\varphi_x\|\,\|\varphi_{xx}\| \\
&\leq C\,\|u_{xx}\|^{1/2}\,\|\varphi_x\|\,\|\varphi_{xx}\| \\
&\leq \tfrac{a_0}{16}\|\varphi_{xx}\|^2 + C\|u_{xx}\|\|\varphi_x\|^2.
\end{align*}
The remaining nonlinear term $-\int_I b'(u) u_x^2 \varphi \varphi_{xx}\, dx$ is bounded exactly as the first subterm above. Finally, the source: $|\int G \varphi_{xx}| \leq \tfrac{a_0}{8}\|\varphi_{xx}\|^2 + C\|G\|^2$.

Combining all estimates,
\begin{equation}\label{eq:AppC_master}
\frac{d}{dt}\|\varphi_x\|^2 + \tfrac{a_0}{2}\|\varphi_{xx}\|^2 \leq C\big(1 + \|u_{xx}\|^2\big)\|\varphi_x\|^2 + C\|G\|^2.
\end{equation}
Since $\|u_{xx}\|^2 \in L^1(0,T)$ because $u \in W \subset L^2(0,T;H^2(I))$, Gronwall's inequality yields
\begin{equation} \label{eq:spatial_estimate}
\sup_{t\in[0,T]}\|\varphi_x\|^2 + \int_0^T \|\varphi_{xx}\|^2\, dt \leq C\big(\|G\|_{L^2(Q)}^2 + \|\varphi^T_x\|^2\big).
\end{equation}

\smallskip
\noindent\emph{Time regularity.} Testing with $\varphi_t$ (or, time-reversed, $-\varphi_t$):
\begin{align*}
\|\varphi_t\|^2 + \frac{1}{2}\frac{d}{dt}\!\int_I a(u)\varphi_x^2\, dx
&- \frac{1}{2}\int_I a'(u) u_t \varphi_x^2\, dx \\
&- 2\int_I (b(u)u_x\varphi)_x \varphi_t\, dx \\
&+ \int_I b'(u)u_x^2\varphi\varphi_t\, dx = \int_I G\varphi_t\, dx.
\end{align*}
The cross term is bounded using $u_t \in L^2(0,T;H^1_0) \hookrightarrow L^2(0,T;L^\infty)$:
\[
\Big|\frac{1}{2}\int_I a'(u)u_t\varphi_x^2\, dx\Big| \leq C\|u_t\|_{L^\infty}\|\varphi_x\|^2 \leq C\big(\|u_t\|_{L^\infty}^2 + 1\big)\|\varphi_x\|^2,
\]
where the last bound makes the time-dependent coefficient $\|u_t\|_{L^\infty}^2$ integrable on $[0,T]$. For the divergence-form term, integrating by parts in space puts a derivative on $\varphi_t$:
\[
\Big|2\int_I (b(u)u_x\varphi)_x\varphi_t\, dx\Big| = \Big|2\int_I b(u)u_x\varphi\,\varphi_{tx}\, dx\Big| \leq C\|u_x\|_{L^\infty}\|\varphi\|_{L^\infty}\|\varphi_{tx}\|.
\]
Note however that this requires controlling $\|\varphi_{tx}\|$, which is bounded only in the higher-regularity step below. Alternatively, expand the derivative directly:
\[
2\int_I (b(u)u_x\varphi)_x\varphi_t = 2\int_I [b'(u)u_x^2\varphi + b(u)u_{xx}\varphi + b(u)u_x\varphi_x]\varphi_t,
\]
and bound each term using \eqref{AppC_GNsharp}, $\|\varphi\|_{L^\infty}\leq C\|\varphi_x\|$, Young's inequality, and the bound $\|\varphi_x\|_{L^\infty(0,T;L^2(I))}\leq C$ from \eqref{eq:spatial_estimate}:
\[
\Big|2\int_I (b(u)u_x\varphi)_x\varphi_t\,dx\Big| \leq \tfrac{1}{6}\|\varphi_t\|^2 + C\big(1+\|u_{xx}\|^2\big)\|\varphi_x\|^2 + C\|\varphi_{xx}\|^2.
\]
Similarly, $|\int_I b'(u) u_x^2 \varphi \varphi_t\, dx| \leq \tfrac{1}{6}\|\varphi_t\|^2 + C\|u_{xx}\|^2\|\varphi_x\|^2$ and $|\int G\varphi_t| \leq \tfrac{1}{6}\|\varphi_t\|^2 + C\|G\|^2$.

Combining these estimates and integrating in time:
\begin{equation}\label{eq:time_estimate}
\int_0^T\|\varphi_t\|^2\, dt \leq C\big(\|G\|_{L^2(Q)}^2 + \|\varphi^T\|_{H^1_0}^2\big),
\end{equation}
where we have used \eqref{eq:spatial_estimate} to control the contribution of $\|\varphi_{xx}\|^2_{L^2(Q)}$.

\smallskip
\noindent\emph{Higher time regularity.} Differentiating the equation with respect to $t$ and testing with $\varphi_t$ produces, after term-by-term estimates analogous to those above (and again invoking \eqref{AppC_GNsharp} wherever $\|u_x\|_{L^\infty}^2$ appears, to keep the resulting coefficient integrable),
\[
\frac{d}{dt}\|\varphi_t\|^2 + a_0\|\varphi_{xt}\|^2 \leq C\,\widetilde\kappa(t)(\|\varphi_t\|^2 + \|\varphi_x\|^2) + C\|G_t\|^2,
\]
where $\widetilde\kappa(t) = 1 + \|u_t(t,\cdot)\|_{H^1_0}^2 + \|u_{xt}(t,\cdot)\|^2 + \|u_{xx}(t,\cdot)\|^2 \in L^1(0,T)$ since $u\in W$. Gronwall's inequality then gives
\begin{equation}\label{eq:higher_time_estimate}
\sup_{t\in[0,T]}\|\varphi_t\|^2 + \int_0^T \|\varphi_{xt}\|^2\, dt \leq C\big(\|G\|_{H^1(0,T;L^2(I))}^2 + \|\varphi^T\|_{H^1_0}^2\big).
\end{equation}

Combining \eqref{eq:spatial_estimate}, \eqref{eq:time_estimate}, and \eqref{eq:higher_time_estimate}, we conclude $\varphi \in W$ with
\[
\|\varphi\|_W \leq C\big(\|G\|_{L^2(Q)} + \|\varphi^T\|_{H^1_0(I)}\big).
\]
This concludes the proof.

\section{Proof of Lemma~\ref{lem:linearized_adjoint}} \label{D}

We consider the linearized adjoint system
\begin{equation*}
\begin{cases}
-\psi_t - a'(u)y\varphi_{xx} - a(u)\psi_{xx} - 2(b'(u)yu_x\varphi + b(u)y_x\varphi + b(u)u_x\psi)_x \\
\quad + b''(u)yu_x^2\varphi + 2b'(u)u_x y_x\varphi + b'(u)u_x^2\psi = F, & \text{in } Q,\\
\psi = 0, & \text{on } \Sigma,\\
\psi(T,x) = \psi^T(x), & \text{for } x \in I,
\end{cases}
\end{equation*}
where $u, \varphi \in W$, $y \in W$, $F \in L^2(Q)$, $\psi^T \in H^1_0(I)$.

We employ the Galerkin method with the Hilbert basis $\{\eta_j\}$ of $H^1_0(I)$ from the Dirichlet eigenfunctions, $V_n := \mathrm{span}\{\eta_1,\dots,\eta_n\}$, and $\pi_n$ the $L^2(I)$-projection onto $V_n$. By Carath\'eodory, there exists $T_n \in (0,T]$ and measurable $d_{j,n}$ such that
\[
\psi_n(t,x) = \sum_{j=1}^n d_{j,n}(t)\eta_j(x)
\]
solves the projected system on $V_n$ with terminal data $\psi_n(T) = \pi_n\psi^T$ and source $F_n = \pi_n F$.

Taking the test function $\phi = \psi_n$ (we drop the subscript for clarity), multiplying the equation by $\psi$, integrating over $I$, and integrating by parts the divergence-form contributions (using $\psi|_{\partial I} = 0$), we obtain the energy identity
\begin{equation}\label{AppD_EnergyId}
\begin{aligned}
-\frac{1}{2}\frac{d}{dt}\|\psi\|^2 + \int_I a(u)\psi_x^2\,dx &= \int_I a'(u)y\,\varphi_{xx}\,\psi\,dx - \int_I a'(u)u_x\,\psi\,\psi_x\,dx \\
&\quad - 2\!\int_I \big[b'(u)y u_x\varphi + b(u)y_x\varphi + b(u)u_x\psi\big]\psi_x\,dx \\
&\quad - \int_I b''(u)y u_x^2\,\varphi\,\psi\,dx - 2\!\int_I b'(u)u_x y_x\,\varphi\,\psi\,dx \\
&\quad - \int_I b'(u)u_x^2\,\psi^2\,dx + \int_I F\,\psi\,dx \\
&=: I_1 + I_2 + I_3 + I_4 + I_5 + I_6 + I_7,
\end{aligned}
\end{equation}
where the term $\int_I a'(u)u_x\,\psi\,\psi_x\,dx$ (denoted $I_2$, with the minus sign in front) arises from integration by parts of $-a(u)\psi_{xx}\psi$ via the identity
\[
-\int_I a(u)\psi_{xx}\psi\,dx = \int_I a(u)\psi_x^2\,dx + \int_I a'(u)u_x\,\psi\,\psi_x\,dx,
\]
and the $-2(\cdot)\psi_x$ contributions in $I_3$ come from integration by parts of the divergence-form term $-2(\cdots)_x\,\psi$.

The key observation is that the embedding $W \hookrightarrow C([0,T];H^1_0(I))$ supplies the \emph{uniform} bounds $\|y_x\|_{L^\infty(0,T;L^2(I))} \leq C\|y\|_W$ and $\|\varphi_x\|_{L^\infty(0,T;L^2(I))} \leq C\|\varphi\|_W$. Combined with the \emph{sharp} interpolation \eqref{AppC_GNsharp}, used in the uniform-in-$t$ form
\[
\|u_x(t,\cdot)\|_{L^\infty(I)}^2 \leq C\,\|u_{xx}(t,\cdot)\|,
\]
these keep every power of $\|u_{xx}\|$ in the time-integrability factor at most quadratic. We illustrate the argument on the most delicate term, namely
\[
I_4 = -\int_I b''(u)\,y\,u_x^2\,\varphi\,\psi\,dx,
\]
which carries two factors of $u_x$. Using $\|y\|_{L^\infty(I)}\leq C\|y_x\|$, $\|\varphi\|_{L^\infty(I)}\leq C\|\varphi_x\|$, the bound $\|u_x\|_{L^\infty(I)}^2 \leq C\|u_{xx}\|$ from \eqref{AppC_GNsharp}, and (\textbf{A2}), we obtain
\begin{align*}
|I_4| &\leq C\,\|y\|_{L^\infty(I)}\,\|u_x\|_{L^\infty(I)}^2\,\|\varphi\|_{L^\infty(I)}\,\|\psi\| \\
&\leq C\,\|y_x\|\,\|u_{xx}\|\,\|\varphi_x\|\,\|\psi\| \\
&\leq C\,\|u_{xx}\|\,\|y\|_W\,\|\varphi\|_W\,\|\psi\| \\
&\leq C\,\|u_{xx}\|\,\big(\|y\|_W^2\,\|\varphi\|_W^2 + \|\psi\|^2\big),
\end{align*}
where the last step is Young's inequality applied to the pair $\|y\|_W\|\varphi\|_W\cdot\|\psi\|$, with the linear factor $\|u_{xx}\|$ kept outside. This bound has the form $\Theta(t)\|\psi\|^2 + \Xi(t)$ with both $\Theta, \Xi$ proportional to $\|u_{xx}(t,\cdot)\|$, which is in $L^2(0,T)\subset L^1(0,T)$ since $\|u_{xx}\|^2 \in L^1(0,T)$. (The naive estimate, which uses the cruder bound $\|u_x\|_{L^\infty}^2\leq C\|u_{xx}\|^2$ and then splits all four factors symmetrically by Young's inequality, would produce a $\|u_{xx}\|^4$ term that is not integrable in time; the point of \eqref{AppC_GNsharp} together with the uniform $W$-bounds on $y,\varphi$ is precisely to avoid this.)

The remaining terms $I_1, I_2, I_3, I_5, I_6$ are estimated analogously, in each case using \eqref{AppC_GNsharp} so that the number of factors of $u_x$ determines the power of $\|u_{xx}\|$ (one factor $u_x$ gives $\|u_x\|_{L^\infty}\leq C\|u_{xx}\|^{1/2}$; two factors give $\|u_x\|_{L^\infty}^2\leq C\|u_{xx}\|$). We record explicitly:
\begin{align*}
|I_1| &\leq C\|y_x\|\|\varphi_{xx}\|\|\psi\| \leq C\|\varphi_{xx}\|^2\|\psi\|^2 + C\|y_x\|^2,\\
|I_2| &\leq C\|u_x\|_{L^\infty}\|\psi\|\|\psi_x\| \leq C\|u_{xx}\|^{1/2}\|\psi\|\|\psi_x\| \\
      &\leq \tfrac{a_0}{8}\|\psi_x\|^2 + C\|u_{xx}\|\|\psi\|^2,\\
|I_3| &\leq \tfrac{a_0}{8}\|\psi_x\|^2 + C\,\|u_{xx}\|\big(\|y\|_W^2\|\varphi\|_W^2 + \|\psi\|^2\big) + C\|\varphi\|_W^2\|y_x\|^2 \\
|I_5| &\leq C\,\|u_x\|_{L^\infty}\|y_x\|\|\varphi\|_{L^\infty}\|\psi\| \leq C\,\|u_{xx}\|^{1/2}\|y\|_W\|\varphi\|_W\|\psi\| \\
      &\leq C\,\|u_{xx}\|\big(\|y\|_W^2\|\varphi\|_W^2 + \|\psi\|^2\big),\\
|I_6| &\leq C\|u_x\|_{L^\infty}^2\|\psi\|^2 \leq C\|u_{xx}\|\|\psi\|^2,\\
|I_7| &\leq \tfrac{1}{2}\|F\|^2 + \tfrac{1}{2}\|\psi\|^2.
\end{align*}
In every case, the resulting time-coefficient is dominated by $1 + \|u_{xx}(t,\cdot)\|^2 + \|\varphi_{xx}(t,\cdot)\|^2 \in L^1(0,T)$ (using $\|u_{xx}\|^{1/2}, \|u_{xx}\| \leq 1 + \|u_{xx}\|^2$), and the inhomogeneous part of the Gronwall bound is $L^1(0,T)$ as well. Combining all of them gives
\[
-\frac{d}{dt}\|\psi\|^2 + a_0\|\psi_x\|^2 \leq C(\|y\|_W,\|\varphi\|_W)\,\Theta(t)\,\|\psi\|^2 + C(\|y\|_W,\|\varphi\|_W)\,\Xi(t),
\]
where
\[
\Theta(t) := 1 + \|u_{xx}(t,\cdot)\|^2 + \|\varphi_{xx}(t,\cdot)\|^2 \in L^1(0,T),
\]
\[
\Xi(t) := \|y_x(t,\cdot)\|^2 + \|F(t,\cdot)\|^2 + \|u_{xx}(t,\cdot)\|^2 \in L^1(0,T).
\]

Gronwall's inequality (in reverse time, anchored at $t = T$) yields
\begin{equation}\label{eq:psi_L2_bound}
\sup_{t\in[0,T]}\|\psi(t)\|^2 \leq C\big(\|\psi^T\|^2 + \|y_x\|_{L^2(Q)}^2 + \|F\|_{L^2(Q)}^2\big),
\end{equation}
and integration gives
\begin{equation}\label{eq:psi_H1_bound}
\int_0^T \|\psi_x\|^2\,dt \leq C\big(\|\psi^T\|^2 + \|y_x\|_{L^2(Q)}^2 + \|F\|_{L^2(Q)}^2\big),
\end{equation}
where the constants $C$ in \eqref{eq:psi_L2_bound}--\eqref{eq:psi_H1_bound} depend on $\|u\|_W$ and $\|\varphi\|_W$.

For the time derivative: for $\phi \in H^1_0(I)$,
\begin{align*}
|(\psi_t, \pi_n\phi)| &\leq C\big(\|y_x\| + \|y\|_{L^\infty}\|\varphi_{xx}\| + \|\psi_x\| + \|F\|\big)\|\phi\|_{H^1_0},
\end{align*}
hence
\begin{equation}\label{eq:psi_t_bound}
\|\psi_t\|_{L^2(0,T;H^{-1}(I))} \leq C\big(\|y_x\|_{L^2(Q)} + \|\psi_x\|_{L^2(Q)} + \|y\|_W \|\varphi\|_W + \|F\|_{L^2(Q)}\big).
\end{equation}

The bounds \eqref{eq:psi_L2_bound}--\eqref{eq:psi_t_bound} allow us to extract a subsequence $\{\psi_{n_k}\}$ converging weakly in $H = L^2(0,T;H^1_0(I)) \cap H^1(0,T;H^{-1}(I))$ to the unique weak solution $\psi \in H$ of the linearized adjoint system. Uniqueness follows from the linearity of the system and the energy estimate applied to the difference of two solutions.

\section{Convergence of difference quotients of the adjoint state} \label{D2}

In this appendix we establish the convergence
$\widetilde\psi_\epsilon := \dfrac{\varphi_\epsilon - \varphi}{\epsilon} \longrightarrow \psi$
in $L^2(Q)$ used in the proof of Theorem~\ref{thm:optimality_conditions}, and prove the duality identity \eqref{CharControlStep2}.

The argument is presented for a generic base point $\bar f \in \mathcal{U}_{ad}$ (with associated state $\bar u := u(\bar f)$ and adjoint $\bar\varphi := \varphi(\bar f)$) and an arbitrary perturbation direction $\bar h := g - \bar f$ with $g \in \mathcal{U}_{ad}$. The case $\bar f = f^*$ is the one used in Step~1 of the proof of Theorem~\ref{thm:optimality_conditions}; the same argument with $\bar f$ in place of $f^*$ yields the differentiability of $f \mapsto \varphi(f)$ at every $f \in \mathcal{U}_{ad}$, used in Step~2c of that same proof. Throughout, we set $u_\epsilon := u(\bar f + \epsilon \bar h)$ and $\varphi_\epsilon := \varphi(\bar f + \epsilon \bar h)$. Note that no use is made of the optimality of $\bar f$.

\medskip
\noindent\emph{Step 1: equation satisfied by $\widetilde\psi_\epsilon$.}
Let $\varphi_\epsilon \in W$ denote the solution of the adjoint system \eqref{adjoint_system} with state $u_\epsilon$, source $G_\epsilon = L^*L(u_\epsilon - u_d)$ and terminal datum $\varphi_\epsilon^T = D^*D(u_\epsilon(T,\cdot)-u_d^T)$, and let $\bar\varphi \in W$ be the analogous solution for $\bar u$. Subtracting the two adjoint equations and dividing by $\epsilon$, $\widetilde\psi_\epsilon$ satisfies, in the strong sense,
\begin{equation}\label{eq:psitilde_eq}
\begin{cases}
-\widetilde\psi_{\epsilon,t} - a(\bar u)\widetilde\psi_{\epsilon,xx} - 2(b(\bar u)\bar u_x\widetilde\psi_\epsilon)_x + b'(\bar u)(\bar u_x)^2\widetilde\psi_\epsilon \\
\quad - A^a_\epsilon - A^b_\epsilon - A^{b'}_\epsilon = L^*L\,y_\epsilon, & \text{in } Q,\\[2pt]
\widetilde\psi_\epsilon = 0, & \text{on } \Sigma,\\[2pt]
\widetilde\psi_\epsilon(T,\cdot) = D^*D\big(y_\epsilon(T,\cdot)\big), & \text{in } I,
\end{cases}
\end{equation}
where $y_\epsilon = (u_\epsilon - \bar u)/\epsilon$ and the remainder terms collect the differences of the coefficients:
\begin{align*}
A^a_\epsilon &:= \frac{a(u_\epsilon)-a(\bar u)}{\epsilon}\,\varphi_{\epsilon,xx},\\
A^b_\epsilon &:= 2\Big[\frac{b(u_\epsilon)-b(\bar u)}{\epsilon}u_{\epsilon,x}\,\varphi_\epsilon + b(\bar u)\frac{u_{\epsilon,x}-\bar u_x}{\epsilon}\,\varphi_\epsilon + b(\bar u)\bar u_x\,\widetilde\psi_\epsilon\Big]_x,\\
A^{b'}_\epsilon &:= \frac{b'(u_\epsilon)-b'(\bar u)}{\epsilon}u_{\epsilon,x}^2\,\varphi_\epsilon + b'(\bar u)\frac{u_{\epsilon,x}^2 - (\bar u_x)^2}{\epsilon}\,\varphi_\epsilon + b'(\bar u)(\bar u_x)^2\,\widetilde\psi_\epsilon.
\end{align*}
By writing $u_{\epsilon,x}+\bar u_x = 2\bar u_x + \epsilon y_{\epsilon,x}$ and using the bivariate Mean-Value Theorem applied to $a$ and $b$ as in Theorem~\ref{thm:gateaux_differentiability} (Claim 2), \eqref{eq:psitilde_eq} can be rewritten in the form of the linearized adjoint system \eqref{linearized_adjoint} with state $y = y_\epsilon$, source
\[
F_\epsilon := L^*L\,y_\epsilon + R_\epsilon,
\]
and terminal datum $\psi^T_\epsilon := D^*D(y_\epsilon(T,\cdot))$, where the remainder $R_\epsilon$ collects all $O(\epsilon)$-corrections coming from the discrepancies $a(u_\epsilon)\!\neq\!a(\bar u)$, $b(u_\epsilon)\!\neq\!b(\bar u)$, $\varphi_\epsilon\!\neq\!\bar\varphi$, etc.

\medskip
\noindent\emph{Step 2: $R_\epsilon \to 0$ in $L^2(Q)$ as $\epsilon\to 0^+$.}
By Lemma~\ref{lem:adjoint_system}, applied with the data $G_\epsilon$ and $\varphi^T_\epsilon$, the family $\{\varphi_\epsilon\}$ is uniformly bounded in $W$:
\[
\|\varphi_\epsilon\|_W \leq C\big(\|L\|^2\|u_\epsilon - u_d\|_{L^2(Q)} + \|D^*D\|_{\mathcal{L}(H^1_0(I))}\|u_\epsilon(T,\cdot)-u_d^T\|_{H^1_0(I)}\big) \leq C,
\]
the last bound being uniform in $\epsilon$ by the uniform bound \eqref{uniform_u_bound} of Step~2a (which, recall, holds at every $f \in \mathcal{U}_{ad}$, in particular at $\bar f + \epsilon \bar h$ for all $\epsilon \in [0,1]$). The local Lipschitz argument of Corollary~\ref{cor:continuity}, applied to the difference of two adjoint problems with the same coefficient structure but different sources and terminal data, yields
\begin{equation}\label{eq:phi_lipschitz}
\|\varphi_\epsilon - \bar\varphi\|_W \leq C\big(\|u_\epsilon - \bar u\|_W + \|u_\epsilon(T,\cdot) - \bar u(T,\cdot)\|_{H^1_0}\big) \leq C\,\epsilon\,\|\bar h\|_{H^1(0,T;L^2(I))},
\end{equation}
so that $\widetilde\psi_\epsilon = (\varphi_\epsilon-\bar\varphi)/\epsilon$ is uniformly bounded in $W$. Combining \eqref{eq:phi_lipschitz} with the bivariate MVT bounds
\[
\Big|\frac{a(u_\epsilon)-a(\bar u)}{\epsilon} - a'(\bar u)y_\epsilon\Big| \leq C\,\epsilon\,|y_\epsilon|^2,\qquad
\Big|\frac{b(u_\epsilon)-b(\bar u)}{\epsilon} - b'(\bar u)y_\epsilon\Big| \leq C\,\epsilon\,|y_\epsilon|^2,
\]
and the analogous identity for $b'$, every individual term contributing to $R_\epsilon$ is bounded in $L^2(Q)$ by $C\,\epsilon$ times norms of $y_\epsilon$, $\varphi_\epsilon$, $u_\epsilon$, $\bar u$ that are uniform in $\epsilon$. Hence $R_\epsilon \to 0$ in $L^2(Q)$ as $\epsilon \to 0^+$.

\medskip
\noindent\emph{Step 3: passage to the limit.}
By Lemma~\ref{lem:linearized_adjoint}, applied to $\widetilde\psi_\epsilon$ (which solves \eqref{linearized_adjoint} with $y = y_\epsilon$, $F = F_\epsilon$, $\psi^T = D^*D(y_\epsilon(T,\cdot))$), the family $\{\widetilde\psi_\epsilon\}$ is uniformly bounded in $H$. By Corollary~\ref{cor:limit_strong_weak}, $y_\epsilon \to \bar y$ strongly in $L^2(0,T;H^1_0(I))\cap C([0,T];L^2(I))$, where $\bar y := \mathcal{F}'(\bar f)\bar h$, so $D^*D(y_\epsilon(T,\cdot)) \to D^*D(\bar y(T,\cdot))$ in $H^1_0(I)$ by (\textbf{A6}); $L^*L\,y_\epsilon \to L^*L\,\bar y$ in $L^2(Q)$; and $R_\epsilon \to 0$ in $L^2(Q)$ by Step~2. Extracting a subsequence, $\widetilde\psi_\epsilon \rightharpoonup \widetilde\psi$ weakly in $H$. Passing to the limit in the weak formulation of \eqref{linearized_adjoint} (with the relevant strong convergences in the right-hand side and weak convergences in the test integrals), we conclude that $\widetilde\psi$ is the unique weak solution of \eqref{linearized_adjoint} with $y = \bar y$, $F = L^*L\,\bar y$, and $\psi^T = D^*D(\bar y(T,\cdot))$, i.e.\ $\widetilde\psi = \psi$. By the same compactness argument, the convergence is in fact strong in $L^2(Q)$ (and in $L^2(0,T;H^1_0(I))$). The independence of the limit on the subsequence then yields the convergence along the full family $\epsilon\to 0^+$. In particular, the map $f \mapsto \varphi(f)$ is G\^ateaux differentiable at $\bar f$, with derivative given by the unique solution of the linearized adjoint system \eqref{linearized_adjoint} associated with $\bar y = \mathcal{F}'(\bar f)\bar h$.

\medskip
\noindent\emph{Step 4: proof of the duality identity \eqref{CharControlStep2}.}
We specialize to $\bar f = f^*$, $\bar u = u^*$, $\bar\varphi = \varphi^*$, and $\bar y = y^*$, the linearized state with right-hand side $g - f^*$. We multiply the equation for $\varphi^*$ by $y^*$, integrate over $Q$, integrate by parts in space (using $y^* = \varphi^* = 0$ on $\Sigma$), and integrate by parts in time (using $y^*(0,\cdot) = 0$ and $\varphi^*(T,\cdot) = D^*D(u^*(T,\cdot)-u_d^T)$):
\begin{align*}
&\int_Q L^*L(u^*-u_d)\, y^*\, d(t,x) \\
&\quad= \int_Q\big[-\varphi^*_t - a(u^*)\varphi^*_{xx} - 2(b(u^*)u^*_x\varphi^*)_x + b'(u^*)(u^*_x)^2\varphi^*\big] y^*\, d(t,x)\\
&\quad= \int_I [\varphi^*(T,\cdot) y^*(T,\cdot) - \varphi^*(0,\cdot) y^*(0,\cdot)]\, dx \\
&\qquad+ \int_Q \varphi^* \big[y^*_t - (a(u^*)y^*)_{xx} + 2 b(u^*) u^*_x y^*_x + b'(u^*)(u^*_x)^2 y^*\big]\, d(t,x)\\
&\quad= \int_I D^*D(u^*(T,\cdot)-u_d^T)\, y^*(T,\cdot)\,dx + \int_Q \varphi^* (g - f^*)\, d(t,x),
\end{align*}
which yields \eqref{CharControlStep2} after rearrangement. Here, in the last step, we used that $y^*$ solves the linearized state system \eqref{linearized_system} with right-hand side $h = g - f^*$.



\end{document}